\documentclass[11pt, reqno]{amsart}
\usepackage{comment}
\usepackage{xcolor}
\usepackage{stmaryrd}
\usepackage{amsfonts}
\usepackage{amssymb}
\usepackage{amsmath,amsthm}
\usepackage{latexsym}
\usepackage{amscd}
\usepackage{esint}
\usepackage{mathrsfs}
\usepackage{dsfont}
\usepackage{setspace}
\usepackage{enumitem}
\usepackage[margin=1.25in]{geometry}
\usepackage{hyperref}

\newtheorem{prop}{Proposition}[section]
\newtheorem{thm}[prop]{Theorem}
\newtheorem{lemm}[prop]{Lemma}
\newtheorem{coro}[prop]{Corollary}

\newtheorem*{claim*}{Claim}

\theoremstyle{definition}

\newtheorem{rmk}[prop]{Remark}

\newcommand{\CC}{\mathbb{C}}

\newcommand{\PP}{\mathbb{P}}

\newcommand{\RR}{\mathbb{R}}

\newcommand{\cC}{\mathcal C}

\newcommand{\cK}{\mathcal K}
\newcommand{\cL}{\mathcal L}

\def\ff{\mathfrak{f}}

\def\fp{\mathfrak{p}}
\def\fsu{\mathfrak{su}}

\DeclareMathOperator{\tr}{tr}

\DeclareMathOperator{\Span}{span}

\DeclareMathOperator{\supp}{supp}

\DeclareMathOperator{\Div}{div}

\DeclareMathOperator{\re}{Re}

\DeclareMathOperator{\Ric}{Ric}

\newcommand{\ep}{\varepsilon}

\newcommand{\pa}[2]{\frac{\partial #1}{\partial #2}}

\newcommand{\paop}[1]{\pa{}{#1}}

\setlist[enumerate]{leftmargin = 2em}
\allowdisplaybreaks

\numberwithin{equation}{section}

\title[Instability of solutions to the Ginzburg--Landau equation]{Instability of solutions to the Ginzburg--Landau equation on $S^{n}$ and $\CC\PP^{n}$}
\author{Da Rong Cheng}
\address{Department of Mathematics, University of Chicago, Chicago, IL 60637}
\email{chengdr@uchicago.edu}

\begin{document}

\maketitle 
\begin{abstract} We study critical points of the Ginzburg--Landau (GL) functional and the abelian Yang--Mills--Higgs (YMH) functional on the sphere and the complex projective space, both equipped with the standard metrics. For the GL functional we prove that on $S^{n}$ with $n \geq 2$ and $\CC\PP^{n}$ with $n \geq 1$, stable critical points must be constants. In addition, for GL critical points on $S^{n}$ for $n \geq 3$ we obtain a lower bound on the Morse index under suitable assumptions. On the other hand, for the abelian YMH functional we prove that on $S^{n}$ with $n \geq 4$ there are no stable critical points unless the line bundle is isomorphic to $S^n \times \CC$, in which case the only stable critical points are the trivial ones. Our methods come from the work of Lawson--Simons.
\end{abstract}


\section{Introduction}
\subsection{Motivation and main results}
A classical theorem of Lawson--Simons~\cite{LS} says that there exist no stable varifolds or currents on the round $S^n$, and that any closed, stable stationary integral current in $\CC\PP^{n}$ with the Fubini--Study metric is an integral combination of complex subvarieties. The idea is to consider the second variation of the volume with respect to ambient deformations generated by special vector fields. In both cases, the vector fields arise as the gradients of eigenfunctions corresponding to the first non-zero eigenvalue of the Laplace operator. 

Since the work of Lawson--Simons, similar methods have been applied to study stable Yang--Mills connections on $S^{n}$ (see for instance~\cite{BL}) and stable harmonic maps on $S^{n}$ (see for instance~\cite{Xin}). In this paper, we apply the ideas in~\cite{LS} to study the stability of solutions on $S^{n}$ and $\CC\PP^{n}$ to certain singularly perturbed elliptic equations related to superconductivity. The motivation comes from the relationship these solutions have with minimal submanifolds of codimension two. Below we introduce the equations of interest and explain their connection to minimal submanifolds.

On a closed Riemannian $n$-manifold $(M^n, g)$, for $\ep > 0$, we consider the following two functionals: The first is the Ginzburg--Landau (GL) functional, given by
\begin{equation}\label{eq:GL-defi}
E_{\ep}(u) = \int_{M}e_{\ep}(u)d\mu_g,\ e_{\ep}(u) =\frac{|\nabla u|^{2}}{2} + \frac{(1 - |u|^{2})^{2}}{4\ep^{2}},
\end{equation}
where $u$ is a complex-valued function on $M$. The Euler--Lagrange equation of $E_{\ep}$ is 
\begin{equation}\label{eq:GL-eq}
\ep^{2}\Delta_g u = (|u|^{2} - 1)u.
\end{equation}
Non-trivial weak solutions in $W^{1, 2} \cap L^{\infty}(M; \CC)$ can be found by min-max methods~\cite{thesis, Ste}, and these are always smooth by elliptic regularity.

The second functional comes from the self-dual abelian Higgs model, and we will refer to it as the abelian Yang--Mills--Higgs (YMH) functional. To set the stage, let $L$ be a complex line bundle over $M$, equipped with a Hermitian metric $\langle\cdot, \cdot\rangle_L$. For a metric connection $D$ on $L$, we let $F_{D}$ denote $\sqrt{-1}$ times its curvature. Then the abelian YMH functional has the form
\begin{equation}\label{eq:YMH-defi}
F_{\ep}(u, D) = \int_{M} e_{\ep}(u, D)d\mu_{g},\ e_{\ep}(u, D) = \ep^{2} |F_D|^{2} + |Du|^{2} + \frac{(1 - |u|^{2})^{2}}{4\ep^{2}},
\end{equation}
where $u: M \to L$ is a section of $L$, and $D$ is a metric connection on $L$. Note that, given a metric connection $D_{0}$, all the other metric connections are given by
\[
D = D_{0} - \sqrt{-1}a,
\]
where $a$ is a real $1$-form on $M$. Also, for a metric connection, the curvature is a purely imaginary-valued $2$-form on $M$, and hence $F_{D}$ is a real-valued $2$-form. The Euler--Lagrange equations for $F_{\ep}$ are given by
\begin{equation}\label{eq:YMH-eq}
\left\{
\begin{array}{cl}
\ep^{2}D^{\ast}Du &= \frac{1}{2}(1 - |u|^{2})u\\
\ep^{2}d^{\ast}F_{D} &= \re\langle\sqrt{-1}u, Du\rangle.
\end{array}
\right.
\end{equation}
The operator $D^{\ast}$ is the dual of $D$ with respect to the metric induced on $\Omega^{p}(L)$ by $\langle\cdot, \cdot\rangle_L$ and $g$. Note that if $D = D_{0} - \sqrt{-1}a$, then the second equation reads
\[
\ep^{2}d^{\ast}F_{D_{0}} + \ep^{2}d^{\ast}da = \re\langle\sqrt{-1}u, Du\rangle,
\]
which is not elliptic for $a$. This is of course due to the gauge invariance of $F_{\ep}$, where
\[
F_{\ep}(u, D) = F_{\ep}\big(e^{\sqrt{-1}\varphi}u, D - \sqrt{-1}d\varphi\big), \text{ for any $\varphi: M \to \RR$.}
\]
Thus weak solutions may not be smooth globally. On the other hand, it is known that under mild assumptions they are locally gauge equivalent to smooth solutions~\cite{PiSt}. 

The geometric interest of $E_{\ep}$ and $F_{\ep}$ stems from the fact that, under suitable energy bounds, sequences of critical points give rise to stationary $(n-2)$-varifolds in a number of different settings. See for instance~\cite{LR, BBO, JS}. Below we mention two results of this type on Riemannian manifolds, one for each functional. 
\begin{thm}[\cite{Che, Ste}]
\label{thm:GL-conv}
Let $(M, g)$ be a closed manifold. Suppose $u_{\ep}$ is a solution to~\eqref{eq:GL-eq} for each $\ep> 0$ such that 
\begin{equation}\label{eq:GL-energybound}
\frac{1}{|\log\ep|}E_{\ep}(u_{\ep}) \leq C < \infty \text{ for all }\ep > 0.
\end{equation}
Then, up to taking a subsequence, there exists a smooth harmonic $1$-form $\psi$ and a stationary rectifiable $(n - 2)$-varifold $V$, such that, as $\ep \to 0$,
\begin{equation}
\label{eq:GL-measure-conv}
\frac{1}{|\log\ep|}e_{\ep}(u_\ep)d\mu_g \to \frac{|\psi|^{2}}{2}d\mu_g + \|V\| \text{ as measures on }M.
\end{equation}
\end{thm}

\begin{thm}[Pigati--Stern, \cite{PiSt}]
\label{thm:YMH-conv}
Let $L$ be a Hermitian line bundle over a closed Riemannian manifold $(M, g)$. For each $\ep > 0$, suppose $(u_\ep, D_\ep)$ is a critical point of $F_{\ep}$, satisfying
\begin{equation}\label{eq:YMH-energybound}
F_{\ep}(u_{\ep}, D_{\ep}) \leq C < \infty \text{ for all }\ep > 0.
\end{equation}
Then, up to taking a subsequence, there exists a stationary \textit{integral} $(n -2)$-varifold $V$ such that, as $\ep \to 0$,
\begin{equation}
\label{eq:YMH-measure-conv}
e_{\ep}(u_{\ep}, D_{\ep}) d\mu_g \to 2\pi \|V\| \text{ as measures on }M.
\end{equation}
\end{thm}
Note that the varifold in Theorem~\ref{thm:YMH-conv} has integer multiplicity, whereas no such claims are made in Theorem~\ref{thm:GL-conv}. Indeed the integrality of the limiting varifold in Theorem~\ref{thm:GL-conv} remains an open problem.

Given results of the above type, which relates the first variations of $E_{\ep}$ and $F_{\ep}$ to that of the volume, it is natural to ask how the second variations and Morse indices of critical points are related. The purpose of the present work is to show that some of the results of Simons~\cite{Sim} and Lawson--Simons~\cite{LS} on stable varifolds/currents in $S^{n}$ and $\CC\PP^{n}$ do have analogues for $E_{\ep}$ and $F_{\ep}$. Our main results are stated below. Throughout this paper we assume that $S^{n}$ is equipped with the round metric, and $\CC\PP^{n}$ the Fubini--Study metric. We begin with results on the GL functional.
\begin{thm}\label{thm:GL-no-stable}
Every stable solution of~\eqref{eq:GL-eq} on $S^{n}$ for $n \geq 2$ is necessarily constant with absolute value $1$, regardless of the value of $\ep$.
\end{thm}

\begin{thm}\label{thm:GL-index-bound}
Suppose $n \geq 3$. For all $ C> 0$, there exists $\ep_{0} > 0$ such that if $u$ is a solution on $S^{n}$ to~\eqref{eq:GL-eq} with $\ep < \ep_{0}$, and if
\begin{equation}\label{eq:E_0-bound}
C^{-1}|\log\ep| \leq E_{\ep}(u) \leq C|\log\ep|,
\end{equation}
then the Morse index of $u$ as a critical point of $E_{\ep}$ is at least $2$.
\end{thm}
\begin{thm}\label{thm:GL-no-stable-CPn}
For $n \geq 1$, every stable solution to~\eqref{eq:GL-eq} on $\CC\PP^{n}$ is necessarily constant with absolute value $1$, regardless of the value of $\ep$.
\end{thm}

For the abelian YMH functional, we prove the following. The class $\cC$ in the statement is defined in Section 5.
\begin{thm}\label{thm:YMH-no-stable}
For $n \geq 4$ and $\ep > 0$, suppose $(u, D) \in \cC$ is a stable weak solution of~\eqref{eq:YMH-eq} on $S^{n}$. Then the bundle $L$ is trivial, and $(u, D)$ is gauge equivalent to $(1, d)$.
\end{thm}

\begin{rmk}
\begin{enumerate}
\item[(1)] The main computations involved in the proofs of Theorems~\ref{thm:GL-no-stable},~\ref{thm:GL-no-stable-CPn} and~\ref{thm:YMH-no-stable} (see Propositions~\ref{prop:2nd-var} and~\ref{prop:YMH-inner-stable}) are rather insensitive of the specific form of the potential term $\frac{(1 - |u|^2)^2}{4\ep^2}$. Hence these results should extend when $\frac{(1 - |u|^2)^2}{4\ep^2}$ is replaced by more general potentials $W(u)$ satisfying appropriate conditions. Theorem~\ref{thm:GL-index-bound}, on the other hand, does depend on the choice of potential, as some of the arguments in~\cite{Che, Ste} do. 
\vskip 1mm
\item[(2)] We believe that Theorem~\ref{thm:YMH-no-stable} holds for $n = 3$ as well. On the other hand, when $n = 2$ and the bundle $L$ is non-trivial, $F_{\ep}$ does have non-trivial stable critical points on $S^2$ (in fact on any compact Riemann surface $\Sigma$) if $\ep$ is not too large. These are given by solutions to the vortex equations:
\begin{equation}\label{eq:vortex-eq}
\left\{
\begin{array}{cl}
D u &= \pm \sqrt{-1}\ast Du,\\
\ep \ast F_D &= \pm\frac{1 - |u|^2}{2\ep}.
\end{array}
\right.
\end{equation}
Existence of solutions is essentially established in~\cite{Brad} and, using different methods, ~\cite{Gar}. These solutions are always stable because $F_{\ep}$ can be rewritten as
\begin{equation}\label{eq:YMH-reduction}
F_{\ep}(u, D) = \int_{\Sigma} \frac{1}{2}|Du\mp \sqrt{-1}\ast Du|^2 + \Big|\ep\ast F_D \mp \frac{1 - |u|^2}{2\ep}\Big|^2 d\mu_g \pm \int_{\Sigma}F_D,
\end{equation}
where the last term equals $2\pi$ times the degree of the bundle $L$ up to sign. By analogy with the work of Bourguignon--Lawson~\cite{BL} on the Yang--Mills functional on $S^4$, we suspect that stable critical points of $F_{\ep}$ on $S^{2}$ are in fact solutions to~\eqref{eq:vortex-eq}.
\end{enumerate}
\end{rmk}

Besides their analogy with~\cite{Sim} and~\cite{LS}, our results also continue a long line of work, going back to perhaps~\cite{Mat, CaHo}, on the instability of non-constant solutions to semi-linear equation/systems. For~\eqref{eq:GL-eq} in particular, Jimbo--Morita~\cite{JiMo} proved that any stable solution with Neumann condition on a convex set in $\RR^n$ is necessarily constant. A similar result was obtained by Jimbo--Sternberg~\cite{JiSt} on convex sets in $\RR^2$, for critical points of~\eqref{eq:YMH-defi} with a slightly more general potential term. Serfaty~\cite{Ser} extended the 2D case of the result of~\cite{JiMo} to simply-connected domains, but assuming in addition that $\ep$ is small. The method in~\cite{Ser} was later adapted by Chen~\cite{ChK} to prove the instability of non-constant solutions to~\eqref{eq:GL-eq} on a class of closed two-dimensional surfaces.

Below we summarize the proofs of our results. As in~\cite{LS}, Theorems~\ref{thm:GL-no-stable} and~\ref{thm:YMH-no-stable} are proved by computing the second derivatives of $E_\ep$ and $F_\ep$ with respect to diffeomorphisms generated by vector fields, and then taking the trace over a specific finite-dimensional space of conformal Killing vector fields on $S^n$. A similar idea also underlies many of the results mentioned above. Crucial to this strategy is relating the usual notion of stability to stability with respect to variations by diffeomorphisms. This is not hard for $E_{\ep}$, but slightly delicate for $F_{\ep}$, mainly because weak solutions may not be globally gauge equivalent to smooth solutions. Nonetheless, we manage to localize the computations and patch things together at the end. 

The proof of Theorem~\ref{thm:GL-index-bound} relies on comparing the second variation formulas for $E_{\ep}$ at $u_{\ep}$ with that of the limit varifold associated to $u_{\ep}$ via Theorem~\ref{thm:GL-conv}, and is similar in spirit to the approach in~\cite{Ser} and~\cite{ChK}. Under the assumption that $b_{1}(M) = 0$, we are able to get a ``convergence up to error term'' result for the second variations of $E_{\ep}$ that complements Theorem~\ref{thm:GL-conv} rather nicely and extends a result of Le~\cite{Le} to manifolds. See Proposition~\ref{prop:2nd-innvar-lim}. Theorem~\ref{thm:GL-index-bound} then follows by letting $M = S^{n}$ and combining a contradiction argument with the classical result of Simons~\cite{Sim}.

The strategy for proving Theorem~\ref{thm:GL-no-stable-CPn} is similar to that for Theorem~\ref{thm:GL-no-stable}, except that the trace is taken over a finite-dimensional space of real holomorphic vector fields on $\CC\PP^{n}$. Here we find that if $u$ is a stable solution, and if $X = \nabla f$ where $f$ is any eigenfunction for the lowest non-zero eigenvalue of $-\Delta$, then $v:= \langle X, \nabla u \rangle$ lies in the kernel of the quadratic form $\delta^2 E_{\ep}(u)$. Combining this with the B\^ochner formula and making suitable choices of $f$ gives the desired result.

The remainder of this paper is organized as follows: In Section 1.2 we recall some notation and terminology. In Section 2 we discuss two different notions of second variation for $E_\ep$. In particular we spell out what we mean by ``stable solutions''. In Section 3 we prove Theorems~\ref{thm:GL-no-stable} and~\ref{thm:GL-index-bound}. In Section 4 we prove Theorem~\ref{thm:GL-no-stable-CPn}. Section 5 is similar to Section 2 but concerns $F_{\ep}$. As opposed to Section 2, most of the calculations in Section 5 are done on domains over which the weak solution is gauge equivalent to a smooth solution. Finally, in Section 6, we specialize to $S^{n}$ and put the local computations in Section 5 together to prove Theorem~\ref{thm:YMH-no-stable}.

\vskip 2mm
\noindent\textbf{Acknowledgments.} I would like to thank Andre Neves for suggesting this problem and for numerous enlightening conversations. Thanks also go to Peter Sternberg for very helpful comments on an earlier version of this paper, and to the referee whose many suggestions significantly improved the exposition of the paper.

\subsection{Notation and terminology}\label{sec:notation}
Suppose we have a Riemannian manifold $(M, g)$. The volume measure induced by $g$ is denoted $\mu_g$. Our curvature convention is the following one:
\[
R_{X, Y}Z = \nabla_X\nabla_Y Z - \nabla_Y\nabla_X Z - \nabla_{[X, Y]}Z,
\]
and sometimes we write $R(X, Y, Z, W)$ for $\langle R_{X, Y}Z, W \rangle$. Note that we will often use pointed brackets $\langle \cdot, \cdot \rangle$ to denote the metric $g$ or any metric it induces on tensors bundles over $M$. For example, on $\wedge^{k}T^{\ast}M$ we write 
\[
\langle dx^{i_1} \wedge \cdots \wedge dx^{i_k},  dx^{j_1} \wedge \cdots \wedge dx^{j_k} \rangle = \det (g^{i_\lambda  j_\mu})_{1 \leq \lambda, \mu \leq k}.
\]
We maintain the use of pointed brackets even when the tensors have values in a complex line bundle $L$ equipped with a Hermitian bundle metric $\langle\cdot, \cdot\rangle_{L}$. For example, on the fiber $\big(\wedge^{k}T^{\ast}M \otimes L\big)_{x}$, if $\sigma, \tau \in L_{x}$,  we write
\[
\langle \sigma dx^{i_1} \wedge \cdots \wedge dx^{i_k}, \tau dx^{j_1} \wedge \cdots \wedge dx^{j_k} \rangle_{L} =\langle\sigma, \tau\rangle_{L}\det (g^{i_\lambda  j_\mu})_{1 \leq \lambda, \mu \leq k}.
\]
We often drop the subscript $L$ in $\langle \cdot, \cdot \rangle_{L}$ if no confusion arises from the omission. Also, we use the following cross product notation:
\begin{equation}\label{eq:cross-product}
u \times v = \re\langle \sqrt{-1}u, v \rangle_{L}, \text{ for }u, v \in L_{x}.
\end{equation}

Given a vector field $X$ on $M$, its (full) divergence is given by
\[
\Div_g X = g^{ij}\langle \nabla_{i}X, \partial_{j} \rangle = \sum_{i = 1}^{n}\langle \nabla_{e_{i}}X, e_{i} \rangle,
\]
where $e_1, \cdots, e_n$ is any orthonormal basis for $T_{x}M$. If $S$ is a $k$-dimensional vector subspace of $T_{x}M$, we define the divergence of $X$ at $x$ along $S$ to be
\[
\Div_{S} X = \sum_{i = 1}^{k}\langle \nabla_{\tau_i}X, \tau_i \rangle,
\]
where $\tau_{1}, \cdots, \tau_{k}$ is any orthonormal basis of $S$. The Laplace operator is given by $\Delta_g u = \Div_g \nabla u = g^{ij}\nabla^{2}_{i, j}u$, and is non-positive definite. 

Next we briefly review the standard metric on the complex projective space. The group $U(1)$ acts on $S^{2n + 1} \subset \CC^{n + 1}$ by isometries via
\[
(z_0, \cdots, z_n) \mapsto (e^{i\theta}z_0, \cdots, e^{i\theta}z_n).
\]
Thus there is a metric $g_{FS}$ on $\CC\PP^{n}$, the quotient of $S^{2n + 1}$ under this $U(1)$-action, such that the projection $(S^{2n + 1}, g_0) \to (\CC\PP^{n}, \frac{1}{4}g_{FS})$ is a Riemannian submersion, where $g_0$ is the round metric with constant curvature $1$. Under this normalization we have
\[
\Ric_{g_{FS}} = \frac{n + 1}{2}g_{FS},
\]
and the first non-zero eigenvalue of $-\Delta_{g}$ is equal to $n + 1$. In the standard coordinate charts, say in $(z_1, \cdots, z_n) \mapsto [\frac{1}{\sqrt{(1 + |z|^{2})}},\frac{z_1}{\sqrt{(1 + |z|^{2})}},\cdots,\frac{z_n}{{\sqrt{(1 + |z|^{2})}}}]$, the Fubini--Study metric has the form
\begin{equation}\label{eq:FS-local}
\langle \paop{z^{i}}, \paop{\bar{z}^{j}} \rangle = \frac{2}{(1 + |z|^{2})^{2}}\Big((1 + |z|^{2})\delta_{ij} - \bar{z}_{i}z_{j}  \Big).
\end{equation}
The metric $g_{FS}$ being K\"ahler, its curvature enjoys the following additional symmetries:
\begin{equation}\label{eq:kahler-curvature}
\langle R_{JX, JY}Z, W \rangle = \langle R_{X, Y}Z, W \rangle = \langle R_{X, Y}JZ, JW \rangle,
\end{equation}
where $J$ denotes the complex structure on $\CC\PP^{n}$. Other notation and terminology is introduced when it is needed.

\section{First and second variations of $E_{\ep}$}
\subsection{Preliminaries}
Suppose $(M, g)$ is a closed Riemannian manifold. We begin by discussing the first and second outer variations of $E_{\ep}$, which should be distinguished from the inner variations to be introduced later. Suppose $u \in W^{1, 2} \cap L^{4}(M; \CC)$ and $v: M \to \CC$ is smooth. The first outer variation of $E_{\ep}$ at $u$ in the direction of $v$ is by definition
\begin{equation}\label{eq:GL-1st-var}
\delta E_{\ep}(u)(v) = \frac{d}{dt}E_{\ep}(u + tv)\vert_{t = 0} = \int_{M} \langle \nabla u, \nabla v \rangle + \frac{|u|^{2} - 1}{\ep^{2}}u \cdot v\  d\mu_g,
\end{equation}
where by $u \cdot v$ we mean the usual inner product on $\CC\simeq \RR^{2}$. Of course, $u \in W^{1, 2} \cap L^{4}$ is a weak solution to~\eqref{eq:GL-eq} if and only if $\delta E_{\ep}(u)(v)= 0$ for all $v$ smooth, in which case Kato's inequality~\cite{ReSi} implies that, distributionally on $M$, 
\[
\Delta |u| \geq \frac{|u|^2 - 1}{\ep^2}|u| \geq -\frac{|u|}{\ep^2}.
\]
Since $|u| \in W^{1, 2}$, De Giorgi--Nash estimates imply that $u$ is bounded, and consequently $u$ is smooth by~\eqref{eq:GL-eq} and standard theory. The maximum principle applied to $|u|^2 - 1$ then gives $|u| \leq 1$ on $M$. Below, by ``solution to~\eqref{eq:GL-eq}'' we always mean a smooth solution.

Next we define the second outer variation of $E_{\ep}$ at a solution $u$ by
\begin{align}\label{eq:GL-2nd-var}
\nonumber \delta^{2}E_{\ep}(u)(v, v) &= \frac{d^{2}}{d t^2}E_{\ep}(u + tv)\vert_{t = 0}\\
& = \int_{M} |\nabla v|^2 + \frac{|u|^{2} - 1}{\ep^{2}}|v|^2 + \frac{2(u \cdot v)^2}{\ep^{2}} d\mu_g.
\end{align}
Since $u$ is bounded, $\delta^2 E_{\ep}(u)$ extends to a symmetric bilinear form on the real Hilbert space $W^{1, 2}(M; \CC)\simeq W^{1, 2}(M; \RR^2)$, with associated linear operator
\[
L_{u}v := -\Delta v + \frac{|u|^{2} - 1}{\ep^{2}}v + \frac{2(u \cdot v) u}{\ep^{2}}.
\]
The operator $L_{u}$ possesses a complete set of $L^2$-orthonormal eigenfunctions with eigenvalues 
\[
\lambda_1 \leq \lambda_2 \leq \cdots \to +\infty.
\]
We say that a solution $u$ to~\eqref{eq:GL-eq} is stable if $\lambda_1 \geq 0$, or equivalently when
\begin{equation}\label{eq:stable-defi}
\delta^{2}E_{\ep}(u)(v, v) \geq 0 \text{ for all }v \in W^{1, 2}(M; \CC).
\end{equation}
The index of $u$ as a critical point of $E_{\ep}$ is defined to be the number (counted with multiplicity) of negative eigenvalues of $L_u$. Moreover, the index of $u$ is at least $k$ if and only if there exists a $k$-dimensional subspace $V$ of $W^{1, 2}(M; \RR^2)$ such that $\delta^{2}E_{\ep}(u)$ restricted to $V$ is negative definite.

There is another way in which we can perform the variations. Assuming that $u: M \to \CC$ is smooth and that $X$ is a vector field on $M$, with $\varphi_{t}$ being the flow it generates, then the expression
\begin{equation}\label{eq:inner-pre-defi}
E_{\ep}(\varphi_{t}^{\ast} u) = \int_{M} \frac{|\nabla \varphi_{t}^{\ast}u|^{2}}{2} + \frac{(1 - |\varphi_{t}^{\ast}u|^{2})^{2}}{4\ep^{2}}d\mu_g
\end{equation}
is smooth in $t$, and we define the first and second inner variations of $E_{\ep}$ to be, respectively,
\[
\delta E_{\ep}(u)(X) = \frac{d}{dt}E_{\ep}(\varphi_{t}^{\ast} u) \vert_{t = 0},
\]
\[
\delta^{2} E_{\ep}(u)(X, X) = \frac{d^{2}}{dt^{2}}E_{\ep}(\varphi_{t}^{\ast} u) \vert_{t = 0}.
\]
Note that we use similar notation for the outer and inner variations. However, since one of them applies to functions and the other to vector fields, there should be no confusion. The precise formulae for the first and second inner variations will be given in Section 3. For now we note the relationship between the inner and outer variations.
\begin{prop}\label{prop:inner-outer}
Let $u: M \to \CC$ be a smooth function and let $X$ be a smooth vector field on $M$. Then we have
\begin{equation}\label{eq:1st-inner-outer}
\delta E_{\ep}(u)(X) = \delta E_{\ep}(u)(\nabla_{X}u),
\end{equation}
\begin{equation}\label{eq:2nd-inner-outer}
\delta^{2}E_{\ep}(u)(X, X) = \delta^{2} E_{\ep}(u)(\nabla_{X}u, \nabla_{X}u) + \delta E_{\ep}(u) (\nabla_{X}\nabla_{X}u).
\end{equation}
\end{prop}
\begin{proof}
The Proposition was proved in~\cite{Le} in the case where $M$ is a domain in Euclidean space, but for a more general class of functionals. When we specialize to $E_{\ep}$, the computation is rather short, so we include it below.

First note the following immediate consequences of the definitions of the first and second inner variations:
\begin{equation}\label{eq:inner-defi-1}
\delta E_{\ep}(u)(X) = \int_{M} \langle du, \frac{d}{dt}d(\varphi_{t}^{\ast}u) \rangle + \frac{|u|^{2} - 1}{\ep^{2}}(u \cdot \frac{d}{dt}\varphi_{t}^{\ast}u) d\mu_g,
\end{equation}
\begin{align}
\nonumber\delta^{2} E_{\ep}(u)(X, X) = \int_{M} & |\frac{d}{dt}d(\varphi_{t}^{\ast})u|^{2} + \langle du, \frac{d^{2}}{dt^{2}}d(\varphi_{t}^{\ast}u) \rangle + \frac{|u|^{2} - 1}{\ep^{2}}|\frac{d}{dt}\varphi_{t}^{\ast}u|^{2} \\
\label{eq:inner-defi-2} &+ \frac{2}{\ep^{2}}(u \cdot \frac{d}{dt}\varphi_{t}^{\ast}u)^{2} + \frac{|u|^{2} - 1}{\ep^{2}}(u \cdot \frac{d^2}{dt^2}\varphi_{t}^{\ast}u)d\mu_g.
\end{align}
Next, since $d(\varphi_{t}^{\ast}u) = \varphi_{t}^{\ast}du$, we have
\[
\frac{d}{dt}d(\varphi_{t}^{\ast}u) = \varphi_{t}^{\ast}(\cL_{X}du) = \varphi_{t}^{\ast}d(\nabla_X u).
\]
Differentiating a second time yields 
\[
\frac{d^{2}}{dt^{2}}d(\varphi_{t}^{\ast}u) \vert_{t = 0} = \cL_{X}\cL_{X}du = d(\nabla_X \nabla_X u).
\]
Performing similar computations for $\varphi_{t}^{\ast}u$, substituting into~\eqref{eq:inner-defi-1} and~\eqref{eq:inner-defi-2}, and comparing with~\eqref{eq:GL-1st-var} and~\eqref{eq:GL-2nd-var}, we are done. 
\end{proof}
The following corollary of Proposition~\ref{prop:inner-outer} is immediate.
\begin{coro}\label{coro:inner-outer-coro}
 If $u : M \to \CC$ is a solution to~\eqref{eq:GL-eq}, then 
\[
\delta^{2}E_{\ep}(u)(X, X) = \delta^{2}E_{\ep}(u)(\nabla_{X}u, \nabla_{X}u),
\]
and thus polarizing $\delta^{2}E_{\ep}(u)$ yields a symmetric bilinear form on the space of vector fields on $M$. Moreover, if $u$ is a stable solution to~\eqref{eq:GL-eq}, then for all vector fields $X$ on $M$ we have 
\[
\delta^{2}E_{\ep}(u)(X, X) \geq 0.
\]
\end{coro}

\subsection{The second inner variation formula of $E_{\ep}$}
In this section we continue to assume that $(M, g)$ is a closed Riemannian manifold. The main purpose is to compute the second inner variation of $E_{\ep}$. The computation has been carried out in~\cite{Le} in the case where $M$ is a subset of $\RR^{n}$ with the Euclidean metric. In general one has to be careful about curvature terms. For the sake of clarity, we single out some of the important facts to be used in computing the second variation in the lemma below.
\begin{lemm}\label{lemm:calculations}
Let $X$ be a smooth vector field on $M$ and let $\varphi_{t}$ denote the flow it generates. Let $g_{t} = \varphi_{-t}^{\ast}g$. (Note that we pull back via $\varphi_{-t}$ rather than $\varphi_t$.) Then we have
\begin{enumerate}
\item[(a)] $\frac{d}{dt}g_{t} = -\varphi_{-t}^{\ast}\cL_{X}g$.
\item[(b)] $\frac{d}{dt}g_t^{ij} = g_{t}^{ik}(\varphi_{-t}^{\ast}\cL_X g)_{kl}g_{t}^{lj}$.
\item[(c)] 
\begin{align*}
&-\frac{d}{dt}\varphi_{-t}^{\ast}(\cL_X g)_{ij}\vert_{t = 0}= (\cL_{X}\cL_{X}g)_{ij}\\
&= \langle \nabla_i \nabla_X X, \partial_j \rangle + \langle \nabla_j \nabla_X X, \partial_i \rangle + \langle R_{X, \partial_i}X, \partial_j \rangle + \langle R_{X, \partial_j}X, \partial_i \rangle + 2\langle \nabla_i X, \nabla_j X \rangle.
\end{align*}
\item[(d)] 
\begin{align*}
&-\frac{d}{dt}(\Div_{g_{t}}X)\big|_{t = 0} = \nabla_{X}(\Div_{g}X)\\
&= \Div_{g}(\nabla_{X}X) - \Ric(X, X) - g^{ij}g^{kl}\langle \nabla_i X, \partial_{k} \rangle\langle \nabla_l X, \partial_j \rangle.
\end{align*}
\end{enumerate}
\end{lemm}
\begin{proof}
Part (a) is standard. See for instance~\cite{Lee}. Part (b) follows immediately from (a) since $g_{t}^{ij}$ is by definition the components of the inverse of $g_{t}$.

The first equality in part (c) is just the definition of the Lie derivative. For the second equality, we compute
\[
(\cL_{X}\cL_{X}g)_{ij} = X (\cL_{X}g)_{ij} - (\cL_{X}g)([X, \partial_i], \partial_j) - (\cL_{X}g)(\partial_i, [X, \partial_j]).
\]
To continue, recall that 
\[
(\cL_{X}g)(Y, Z) = \langle \nabla_{Y}X, Z \rangle + \langle \nabla_{Z}X, Y\rangle,
\]
and hence the right-hand side of the previous equality equals
\begin{align}
\nonumber & X \big(\langle \nabla_i X, \partial_j \rangle + \langle \nabla_j X, \partial_i \rangle \big) - \langle \nabla_{[X, \partial_i]}X, \partial_j \rangle - \langle \nabla_j X, [X, \partial_i] \rangle\\
\nonumber & - \langle \nabla_i X, [X, \partial_j] \rangle - \langle \nabla_{[X, \partial_j]}X, \partial_i \rangle\\
\nonumber=& \langle \nabla_X \nabla_i X, \partial_j \rangle + \langle \nabla_X \nabla_j X, \partial_i \rangle + 2\langle \nabla_i X, \nabla_j X \rangle\\
\nonumber & - \langle \nabla_{[X, \partial_i]}X, \partial_j \rangle - \langle \nabla_{[X, \partial_j]}X, \partial_i \rangle\\
=&\langle \nabla_i \nabla_X X, \partial_j \rangle + \langle \nabla_j \nabla_X X, \partial_i \rangle + \langle R_{X, \partial_i}X, \partial_j \rangle + \langle R_{X, \partial_j}X, \partial_i \rangle + 2\langle \nabla_i X, \nabla_j X \rangle.
\end{align}

For part (d), the first equality can be verified using either local coordinates or integration by parts. Specifically, let $f$ be an arbitrary smooth function on $M$. Then we of course have
\[
\int_{M} f \Div_{g_t}X d\mu_{g_t} = -\int_{M} X(f) d\mu_{g_t}.
\]
Next we differentiate both sides with respect to $t$ under the integral and set $t = 0$ to obtain
\[
\int_{M} \big[f \frac{d}{dt}(\Div_{g_t}X)\big|_{t =0} - f ( \Div_{g}X )^{2}\big] d\mu_{g} = \int_{M} X(f)\Div_{g}X d\mu_{g}.
\]
Integrating by parts on the right-hand side yields
\[
\int_{M} X(f)\Div_{g}X d\mu_{g} = -\int_{M} f \Div_{g}( X \Div_{g}X ) d\mu_{g} = -\int_{M} f ( \Div_{g}X)^2 + f \nabla_{X}\Div_{g}X d\mu_g.
\]
Substituting into the previous equality and observing a cancellation, we get
\[
\int_{M} f \frac{d}{dt}(\Div_{g_t}X)\big|_{t =0}  d\mu_{g} = -\int_{M}f \nabla_{X}\Div_{g}X d\mu_{g}.
\]
Since $f$ is arbitrary, the first equality in (d) is proved.

Next, letting $\theta_X$ denote the $1$-form dual to $X$, we have  $(\nabla \theta_{X})_{ij} = \langle \nabla_{i}X, \partial_{j} \rangle, \text{ and }\Div_{g}X = \langle g, \nabla \theta_{X} \rangle$. Thus
\begin{align*}
\nabla_{X}\Div X &= \nabla_{X} \langle g, \nabla \theta_X \rangle = \langle g, \nabla_{X}\nabla \theta_{X} \rangle\\
&= g^{ij}\big( X(\nabla \theta_{X})_{ij} - (\nabla \theta_{X})_{\nabla_{X}\partial_i, \partial_j} - (\nabla \theta_{X})_{\partial_i, \nabla_{X}\partial_j}\big)\\
&= g^{ij}\big( \langle \nabla_{X}\nabla_i X, \partial_j \rangle + \langle \nabla_i X, \nabla_{X}\partial_j \rangle - \langle \nabla_{\nabla_{X}\partial_i}X, \partial_j \rangle - \langle \nabla_i X, \nabla_X \partial_j \rangle \big).
\end{align*}
Cancelling the second and fourth term in the last line and introducing the curvature to switch the order of derivative, we get
\begin{align*}
\nabla_{X}\Div X&= g^{ij}\big( \langle \nabla_i \nabla_{X}X, \partial_{j} \rangle - \langle \nabla_{\nabla_i X}X, \partial_j \rangle  + \langle R_{X, \partial_i}X, \partial_j \rangle\big)\\
&= \Div_{g}(\nabla_{X}X) - \Ric(X, X) - g^{ij}g^{kl}\langle \nabla_i X, \partial_{k} \rangle\langle \nabla_l X, \partial_j \rangle.
\end{align*}
\end{proof}
\begin{prop}\label{prop:2nd-var}
Suppose $X$ is any smooth vector field on $M$ and denote by $\{\varphi_{t}\}$ the flow generated by $X$. For any smooth function $u : M \to \CC$, we have
\begin{align}
&\delta E_{\ep}(u)(X) = -\int_{M} e_{\ep}(u) \Div X  - \langle \nabla_{\nabla u} X, \nabla u \rangle d\mu. \label{eq:1st-innervar}\\
\label{eq:2nd-innervar}  &\delta^{2} E_{\ep}(u)(X, X)\\
\nonumber =& \int_{M} e_{\ep}(u)\big( (\Div X)^{2} - \Ric(X, X) - \langle \nabla_{e_{i}}X, e_{j} \rangle \langle \nabla_{e_{j}}X, e_{i} \rangle + \Div \nabla_X X \big) d\mu\\ 
\nonumber & -\int_{M} 2\langle \nabla_{\nabla u}X, \nabla u \rangle  \Div X  - R(\nabla u, X, X, \nabla u) + |\nabla_{\nabla u}X|^{2} + \langle \nabla_{\nabla u}\nabla_{X} X, \nabla u \rangle d\mu\\
\nonumber & + \int_{M}| \cL_{X}g \llcorner \nabla u |^{2} d\mu.
\end{align}
The $e_{1}, \cdots, e_{n}$ in~\eqref{eq:2nd-innervar} is any orthonormal basis at the point where the integrand is being computed. Note that this choice does not affect the result.
\end{prop}
\begin{proof}
Writing $E_{\ep}(\varphi_{t}^{\ast}u) = \int_{M}e_{\ep}(\varphi_{t}^{\ast}u, g) d\mu_{g}$ to emphasize the dependence on the metric, and denoting $g_{t} = \varphi_{-t}^{\ast}g$, we note that 
\[
E_{\ep}(\varphi_{t}^{\ast}u) = \int_{M} e_{\ep}(u, g_t) d\mu_{g_t} = \int_{M} \frac{1}{2}g_{t}^{ij}\partial_i u \cdot \partial_j u + \frac{(1 - |u|^{2})^{2}}{4\ep^{2}} d\mu_{g_{t}}.
\]
Differentiating under the integral sign, we get
\begin{equation}\label{eq:proof-2nd-var-1}
\frac{d}{dt}E_{\ep}(\varphi_{t}^{\ast}u) = \int_{M} \big[\frac{1}{2} g_{t}^{ik}\big(\varphi_{-t}^{\ast}\cL_{X}g \big)_{kl}g_{t}^{jl}\partial_i u \cdot \partial_j u - e_{\ep}(u, g_t) \Div_{g_{t}}X \big] d\mu_{g_t}.
\end{equation}
Evaluating at $t = 0$ gives~\eqref{eq:1st-innervar} at once. Next, we differentiate~\eqref{eq:proof-2nd-var-1} again to get
\begin{align}
\nonumber \frac{d^{2}}{dt^{2}}E_{\ep}(\varphi_{t}^{\ast}u)\vert_{t = 0} =& \int_{M}\frac{1}{2}\partial_i u \cdot \partial_j u \big[(-\cL_{X}\cL_{X}g)^{ij} + (\cL_{X}g)^{ik}(\cL_{X}g)_{kl}g^{jl} + g^{ik}(\cL_{X}g)_{kl}(\cL_{X}g)^{jl}\big]\\
\nonumber &- \frac{1}{2}g^{ik}(\cL_{X}g)_{kl}g^{lj}( \partial_i  u  \cdot \partial_j u) \Div_{g}X - \frac{d}{dt}e_{\ep}(u, g_t)|_{t = 0}\Div_{g}X\\
\label{eq:proof-2nd-var-2}& - e_{\ep}(u, g) \frac{d}{dt}(\Div_{g_t} X)\big|_{t=0} + e_{\ep}(u, g)(\Div_g X)^{2} d\mu_g.
\end{align}
Note that the two terms in the second line combine to $-g^{ik}(\cL_{X}g)_{kl}g^{lj}( \partial_i  u  \cdot \partial_j u) \Div_{g}X$. To continue, we use Lemma~\ref{lemm:calculations}(d) to replace the term $\frac{d}{dt}(\Div_{g_{t}}X)\big|_{t=0}$ and see that~\eqref{eq:proof-2nd-var-2} becomes
\begin{align}
\nonumber&\frac{d^{2}}{dt^{2}}E_{\ep}(\varphi_{t}^{\ast}u)\vert_{t = 0}\\
\nonumber =& \int_{M} e_{\ep}(u, g)\big[ (\Div_g X)^{2} - \Ric(X, X) - g^{ij}g^{kl}\langle \nabla_i X, \partial_{k} \rangle\langle \nabla_l X, \partial_j \rangle + \Div_{g}(\nabla_X X)  \big]d\mu_g\\
\nonumber & -\int_{M} g^{ik}(\cL_{X}g)_{kl}g^{lj}( \partial_i  u \cdot \partial_j u) \Div_{g}X d\mu_g - \int_{M}\frac{1}{2} (\cL_{X}\cL_{X}g)^{ij} (\partial_i u \cdot \partial_j u) d\mu_g \\
& + \int_{M} (\cL_{X}g)^{ik}(\cL_{X}g)_{kl}g^{jl} (\partial_i u\cdot \partial_j u)d\mu_g \label{eq:calculations-1}
\end{align}
To continue, note that 
\begin{equation}
\int_{M} g^{ik}(\cL_{X}g)_{kl}g^{lj} (\partial_i  u\cdot\partial_j u) \Div_{g}X d\mu_g = \int_{M} (\cL_{X}g)(\nabla u, \nabla u) \Div_g X d\mu_g = 2\int_{M} \langle \nabla_{\nabla u}X, \nabla u \rangle \Div X d\mu_g.
\end{equation}
Next, by Lemma~\ref{lemm:calculations}(c), we have
\begin{equation}
 \int_{M}\frac{1}{2} (\cL_{X}\cL_{X}g)^{ij} \partial_i u \cdot \partial_j u\  d\mu_g = \int_{M} \langle \nabla_{\nabla u}(\nabla_X X), \nabla u \rangle - R(\nabla u, X, X, \nabla u) + |\nabla_{\nabla u}X|^{2}.
\end{equation}
Finally, for the last line of~\eqref{eq:calculations-1}, we have
\begin{equation}
 \int_{M} (\cL_{X}g)^{ik}(\cL_{X}g)_{kl}g^{jl} \partial_i u \cdot\partial_j u\ d\mu_g = \int_{M} |\cL_{X}g\ \llcorner \nabla u|^{2}d\mu_g.
\end{equation}
Putting the three equations above back into~\eqref{eq:calculations-1}, and rewriting
\[
g^{ij}g^{kl}\langle \nabla_i X, \partial_{k} \rangle\langle \nabla_l X, \partial_j \rangle = \langle \nabla_{e_{i}}X, e_{j} \rangle \langle \nabla_{e_{j}}X, e_{i} \rangle 
\]
using an orthonormal frame, we obtain the desired result.
\end{proof}

Recalling the relationship between inner and outer variations noted in Section 2.1, we derive the following result from the previous proposition.
\begin{coro}\label{cor:outerinner}
Suppose $u : M \to \CC$ is a solution to~\eqref{eq:GL-eq} and let $X$, $\varphi_{t}$ be as in the previous Proposition. Then we have
\begin{align}
\nonumber \delta^{2}E_{\ep}(u)(X, X)=& \int_{M} e_{\ep}(u)\big( (\Div X)^{2} - \Ric(X, X) - \langle \nabla_{e_{i}}X, e_{j} \rangle \langle \nabla_{e_{j}}X, e_{i} \rangle \big) d\mu\\ 
\nonumber & -\int_{M} 2\langle \nabla_{\nabla u}X, \nabla u \rangle  \Div X  - R(\nabla u, X, X, \nabla u) +  |\nabla_{\nabla u}X|^{2} d\mu\\
\label{eq:2nd-innervar-solution}  & + \int_{M}| \cL_{X}g \ \llcorner \nabla u |^{2} d\mu.
\end{align}
\end{coro}
\begin{proof}
By~\eqref{eq:1st-inner-outer} and ~\eqref{eq:1st-innervar} with $\nabla_{X} X$ in place of $X$, and recalling that $u$ is a smooth solution, we get
\[
\int_{M} e_{\ep}(u) \Div_g (\nabla_{X}X) - \langle \nabla_{\nabla u}\nabla_{X}X, \nabla u \rangle d\mu_{g} = 0.
\]
Combining this with~\eqref{eq:2nd-innervar} gives the result.
\end{proof}

Combining \eqref{eq:2nd-innervar-solution} with the analysis in~\cite{Che, Ste} on the asymptotic behavior of solutions as $\ep \to 0$ allows us to compare the second variation of $E_{\ep}$ with that of the volume, at least in the case where $b_{1}(M) = 0$. Interestingly, the result is very similar to that in~\cite[Theorem 1.5]{Le}, where the domain is a subset of $\RR^{n}$.
\begin{prop}\label{prop:2nd-innvar-lim}
Suppose $b_{1}(M) = 0$ and that $u_{\ep}$ is a family of solutions to~\eqref{eq:GL-eq} satisfying~\eqref{eq:GL-energybound}. Let $V$ denote the stationary rectifiable $(n-2)$-varifold obtained by applying Theorem~\ref{thm:GL-conv} to the sequence $u_{\ep}$. Then up to taking a subsequence as in Theorem~\ref{thm:GL-conv}, for any smooth vector field $X$ on $M$ we have
\begin{align}\nonumber
&\lim_{k \to \infty} \frac{1}{|\log \ep_k|}\delta^{2}E_{\ep_k}(u_{\ep_k})(X, X)\\ =\ & \delta^{2}V(X, X) + \int_{\Sigma} \big(\langle \nabla_{\nu_1}X, \nu_2 \rangle + \langle \nabla_{\nu_2}X, \nu_1 \rangle\big)^{2} + \big(\langle \nabla_{\nu_1}X, \nu_1 \rangle - \langle \nabla_{\nu_2}X, \nu_2 \rangle\big)^{2}  d\|V\|\label{eq:2nd-innvar-lim},
\end{align}
where $\nu_{1}, \nu_{2}$ is an orthonormal basis for $T_{x}^{\perp}\Sigma$, the orthogonal complement of $T_{x}\Sigma$ in $T_{x}M$. 
\end{prop}
\begin{rmk}
Note that the term
\[
\big(\langle \nabla_{\nu_1}X, \nu_2 \rangle + \langle \nabla_{\nu_2}X, \nu_1 \rangle\big)^{2} + \big(\langle \nabla_{\nu_1}X, \nu_1 \rangle - \langle \nabla_{\nu_2}X, \nu_2 \rangle\big)^{2}
\]
is independent of the choice of $\nu_1, \nu_2$.
\end{rmk}
\begin{proof}[Proof of Proposition \ref{prop:2nd-innvar-lim}]
The second variation formula for a general varifold in a Riemannian manifold can be found in~\cite[p.435]{LS}. Applying it to the stationary rectifiable $(n-2)$-varifold $V$ and expressing the result in our present notation, we arrive at
\begin{align}
\delta^{2}V(X, X) = \int_{\Sigma} \big(\Div_{T\Sigma}X\big)^{2} - &\sum_{i, j = 1}^{n-2}\langle \nabla_{\tau_{i}}X, \tau_{j} \rangle \langle \nabla_{\tau_{j}}X, \tau_{i} \rangle \label{eq:varifold-2nd-var}\\
\nonumber &+ \sum_{i = 1}^{n-2} \big| (\nabla_{\tau_i}X)^{\perp} \big|^{2} - \sum_{i = 1}^{n-2}R(X, \tau_i, \tau_i, X) d\|V\|,
\end{align}
where $\tau_{1}, \cdots, \tau_{n-2}$ is any orthonormal basis of $T_{x}\Sigma$ whenever the latter exists, and $(\cdot)^{\perp}$ denotes the orthogonal projection onto $T_{x}^{\perp}\Sigma$.

By Theorem~\ref{thm:GL-conv}, since $b_1(M) = 0$ by assumption, the harmonic one form $\psi$ must be identically zero, and we have
\[
\frac{1}{|\log\ep_{k}|}e_{\ep_{k}}(u_{\ep_{k}}) d\mu_{g} \to \|V\| \text{ as Radon measures on }M.
\]
Moreover, in~\cite[Section 7]{Che} it is shown that on each geodesic ball $B$ with local orthogonal frame $e_{1}, \cdots, e_{n}$, there exist $\|V\|$-measurable functions $A_{ij}$ such that 
\[
\frac{1}{|\log\ep_{k}|}\nabla_{e_{i}}u_{k} \cdot \nabla_{e_j}u_{k}d\mu_g \to A_{ij} d\|V\| \text{ as measures on $B$},
\]
where the matrix $I - A(x)$ projects orthogonally onto $T_x \Sigma$ for $\|V\|$-a.e. $x \in B$. Consequently, $A(x)$ projects onto $T_{x}^{\perp}\Sigma$ for $\|V\|$-a.e. $x \in B$. Moreover, letting $\omega_1, \cdots, \omega_n$ be the coframe dual to $e_1, \cdots, e_n$, we see that the tensor
\[
A_{ij}\omega^i\otimes \omega^j
\]
is independent of the choice of frame $e_1, \cdots, e_n$. Thus the $A_{ij}$'s obtained on different geodesic balls patch together to define a global object on $M$.

We now divide~\eqref{eq:2nd-innervar-solution} by $|\log\ep_{k}|$ and let $k$ tend to infinity. By the above discussion, we see that 
\begin{align}
\nonumber &\lim_{k \to \infty} \frac{1}{|\log \ep_k|}\delta^{2}E_{\ep_k}(u_{\ep_k})(X, X)\\
\nonumber =& \int_{M}\big( (\Div X)^{2} - \Ric(X, X) - \langle \nabla_{e_{i}}X, e_{j} \rangle \langle \nabla_{e_{j}}X, e_{i} \rangle \big) d\|V\|\\
\nonumber&-\int_{M} \big(2 \langle \nabla_{e_{i}}X, e_{j} \rangle \Div X - R(e_{i}, X, X, e_{j})+ \langle \nabla_{e_{i}}X, \nabla_{e_{j}}X \rangle\big) A_{ij} d\|V\|\\
\label{eq:GL-2ndvar-lim-1}&+ \int_{M}(\cL_{X}g)_{e_{i}, e_{k}} (\cL_{X}g)_{e_{j}, e_{k}} A_{ij}d\|V\|.
\end{align}
At a point $x \in \Sigma$ where $T_{x}\Sigma$ exists, since the integrands above do not depend on the choice of orthonormal basis $e_{1}, \cdots, e_{n}$, we may assume that $e_{1}, \cdots, e_{n- 2} $ span $T_{x}\Sigma$, while $\nu_1 := e_{n-1}$ and $\nu_{2}:= e_{n}$ form a basis for $T_{x}^{\perp}\Sigma$. Recalling that $A$ projects orthogonally onto $T_{x}^{\perp}\Sigma$, we can rewrite the various terms in the integrands above as follows:
\begin{align*}
(\Div X)^{2} &= (\Div_{T\Sigma}X + \Div_{T^{\perp}\Sigma}X)^{2}\\
 A_{ij}\langle \nabla_{e_{i}}X, e_{j} \rangle \Div X  &= (\Div_{T^{\perp}\Sigma}X) (\Div_{T\Sigma}X + \Div_{T^{\perp}\Sigma}X)\\
A_{ij}R(e_i, X, X, e_j) =& R(\nu_1, X, X, \nu_1) + R(\nu_2, X, X, \nu_2)\\
A_{ij}\langle \nabla_{e_i}X, \nabla_{e_j}X \rangle =& |\nabla_{\nu_{1}}X|^{2} + |\nabla_{\nu_{2}}X|^{2}\\
\sum_{k = 1}^{n}(\cL_{X}g)_{e_{i}, e_{k}} (\cL_{X}g)_{e_{j}, e_{k}} A_{ij} =& \sum_{i = 1}^{2}\sum_{k = 1}^{n}(\langle \nabla_{\nu_{i}}X, e_{k} \rangle + \langle \nabla_{e_k}X, \nu_i \rangle)^{2} \nonumber\\
=& |\nabla_{\nu_{1}}X|^{2} + |\nabla_{\nu_{2}}X|^{2} + \sum_{k = 1}^{n}\big|(\nabla_{e_k}X)^{T^{\perp}\Sigma}\big|^{2} \nonumber\\&+ 2\sum_{i = 1}^{2}\sum_{k = 1}^{n}\langle \nabla_{\nu_i}X, e_{k} \rangle\langle \nabla_{e_k}X, \nu_i \rangle.
\end{align*}
Putting these back into~\eqref{eq:GL-2ndvar-lim-1} and noticing some cancellations, we get
\begin{align*}
&\lim_{k \to \infty} \frac{1}{|\log \ep_k|}\delta^{2}E_{\ep_k}(u_{\ep_k})(X, X)\\
=& \int_{M} (\Div_{T\Sigma}X)^{2} - (\Div_{T^{\perp}\Sigma}X)^{2} - \sum_{i= 1}^{n-2}R(e_i, X, X, e_i) + \sum_{k = 1}^{n}\big|(\nabla_{e_k}X)^{T^{\perp}\Sigma}\big|^{2}\\
&+ 2\sum_{i = 1}^{2}\sum_{k = 1}^{n}\langle \nabla_{\nu_i}X, e_{k} \rangle\langle \nabla_{e_k}X, \nu_i \rangle - \sum_{i, j= 1}^{n}\langle \nabla_{e_{i}}X, e_{j} \rangle \langle \nabla_{e_{j}}X, e_{i} \rangle d\|V\|.
\end{align*}
To continue, note that 
\begin{align*}
-(\Div_{T^{\perp}\Sigma}X)^{2} =& -\langle \nabla_{\nu_1}X, \nu_1 \rangle^{2} - \langle \nabla_{\nu_2}X, \nu_2 \rangle^{2} - 2\langle \nabla_{\nu_1}X, \nu_1 \rangle\langle \nabla_{\nu_2}X, \nu_2 \rangle\\
 \sum_{k = 1}^{n}\big|(\nabla_{e_k}X)^{T^{\perp}\Sigma}\big|^{2} =&  \sum_{k = 1}^{n-2}\big|(\nabla_{e_k}X)^{T^{\perp}\Sigma}\big|^{2} + \sum_{i, j = 1}^{2}\langle \nabla_{\nu_i}X, \nu_j \rangle^{2}\\
2\sum_{i = 1}^{2}\sum_{k = 1}^{n}\langle \nabla_{\nu_i}X, e_{k} \rangle\langle \nabla_{e_k}X, \nu_i \rangle =& 2\sum_{i = 1}^{2}\sum_{k = 1}^{n-2}\langle \nabla_{\nu_i}X, e_{k} \rangle\langle \nabla_{e_k}X, \nu_i \rangle + 2\sum_{i,j = 1}^{2}\langle \nabla_{\nu_i}X, \nu_j \rangle\langle \nabla_{\nu_j}X, \nu_i \rangle\\
-\sum_{i, j = 1}^{n}\langle \nabla_{e_{i}}X, e_{j} \rangle \langle \nabla_{e_{j}}X, e_{i} \rangle  =& -\sum_{i, j = 1}^{n-2}\langle \nabla_{e_{i}}X, e_{j} \rangle \langle \nabla_{e_{j}}X, e_{i} \rangle  - 2\sum_{i = 1}^{n-2}\sum_{j = 1}^{2}\langle \nabla_{e_{i}}X, \nu_{j} \rangle \langle \nabla_{\nu_{j}}X, e_{i} \rangle \\
&- \sum_{i, j = 1}^{2}\langle \nabla_{\nu_{i}}X, \nu_{j} \rangle \langle \nabla_{\nu_{j}}X, \nu_{i} \rangle 
\end{align*}
Substituting and making some cancellations, we arrive at
\begin{align*}
&\lim_{k \to \infty} \frac{1}{|\log \ep_k|}\delta^{2}E_{\ep_k}(u_{\ep_k})(X, X)\\
=& \int_{M} \big[(\Div_{T\Sigma}X)^{2}  - \sum_{i= 1}^{n-2}R(e_i, X, X, e_i) + \sum_{k = 1}^{n-2}\big|(\nabla_{e_k}X)^{T^{\perp}\Sigma}\big|^{2} - \sum_{i, j = 1}^{n-2}\langle \nabla_{e_{i}}X, e_{j} \rangle \langle \nabla_{e_{j}}X, e_{i} \rangle \big]d\|V\|\\
&+ \int_{M} \sum_{i, j = 1}^{2}\langle \nabla_{\nu_i}X, \nu_j \rangle\langle \nabla_{\nu_j}X, \nu_i \rangle + \langle \nabla_{\nu_1}X, \nu_2 \rangle^{2} + \langle \nabla_{\nu_2}X, \nu_1  \rangle^{2}  - 2\langle \nabla_{\nu_1}X, \nu_1 \rangle\langle \nabla_{\nu_2}X, \nu_2  \rangle d\|V\|\\
=&  \int_{M} \big[(\Div_{T\Sigma}X)^{2}  - \sum_{i= 1}^{n-2}R(e_i, X, X, e_i) + \sum_{k = 1}^{n-2}\big|(\nabla_{e_k}X)^{T^{\perp}\Sigma}\big|^{2} - \sum_{i, j = 1}^{n-2}\langle \nabla_{e_{i}}X, e_{j} \rangle \langle \nabla_{e_{j}}X, e_{i} \rangle \big]d\|V\|\\
&+ \int_{M} \big(\langle \nabla_{\nu_1}X, \nu_2 \rangle + \langle \nabla_{\nu_2}X, \nu_1 \rangle\big)^{2} + \big(\langle \nabla_{\nu_1}X, \nu_1 \rangle - \langle \nabla_{\nu_2}X, \nu_2 \rangle\big)^{2} d\|V\|.
\end{align*}
Recalling~\eqref{eq:varifold-2nd-var}, we are done.
\end{proof}

\section{Stability and index of Ginzburg--Landau solutions on $S^{n}$}
As in~\cite{LS}, the conformal Killing vector fields of $S^{n}$ which are orthogonal to the Killing fields plays an important role in studying the stability and index of solutions to the Ginzburg--Landau equations. The proposition below summarizes some of the well-known properties of these vector fields. 
\begin{prop}\label{prop:confKilling}
For $\xi \in \RR^{n + 1}$, define $f_{\xi}:S^{n} \to \RR$ by $f_{\xi}(x) = \langle x, \xi \rangle$ and let $X_{\xi} = \nabla f_{\xi}$, where $\nabla$ denotes the covariant derivative on $S^{n}$. Then we have
\begin{enumerate}
\item[(a)] $X_{\xi}(x) = \xi - f_{\xi}(x)x$ for all $x \in S^{n}$. Consequently $|X_{\xi}|^{2} = |\xi|^2 - f_{\xi}^{2}$.
\item[(b)] $\langle \nabla_{v}X_{\xi}, w \rangle = -f_{\xi}(x) \langle v, w \rangle \text{ for all }x \in S^{n},\ v, w \in T_{x}S^{n}$.
\item[(c)] $\Div X_{\xi} = -nf_{\xi}$.
\end{enumerate}
\end{prop}
Now suppose $u: S^{n} \to \CC$ is a solution to~\eqref{eq:GL-eq}. Using Corollary~\ref{cor:outerinner} with $M = S^{n}$ and $X = X_{\xi}$ for any $\xi \in S^n$, along with Proposition~\ref{prop:confKilling}, we get
\begin{align}
\nonumber \delta^{2}E_{\ep}(u)(X_\xi, X_\xi)=& \int_{M} e_{\ep}(u)\big( (\Div X_{\xi})^{2} - \Ric(X_{\xi}, X_{\xi}) - \langle \nabla_{e_{i}}X_{\xi}, e_{j} \rangle \langle \nabla_{e_{j}}X_{\xi}, e_{i} \rangle \big) d\mu\\ 
\nonumber & -\int_{M} 2\langle \nabla_{\nabla u}X_{\xi}, \nabla u \rangle  \Div X_{\xi}  - R(\nabla u, X_{\xi}, X_{\xi}, \nabla u) + |\nabla_{\nabla u}X_{\xi}|^{2} d\mu\\
\nonumber & + \int_{M}| \iota_{\nabla u}\cL_{X_\xi}g |^{2} d\mu\\
\nonumber =&\int_{M} e_{\ep}(u)\big( (n^{2}f_{\xi}^{2} - (n-1)(1 - f_{\xi}^{2}) - nf_{\xi}^{2} \big) d\mu\\ 
\nonumber & -\int_{M} 2n|\nabla u|^{2}f_{\xi}^{2}  - (1 - f_{\xi}^{2})|\nabla u|^{2} + |\langle \nabla u, X_{\xi} \rangle|^{2} +  f_{\xi}^{2}|\nabla u|^{2} d\mu\\
 & + 4\int_{M} f_{\xi}^{2}|\nabla u|^{2} d\mu. \label{eq:2nd-innvar-conKill}
\end{align}

We are now ready to prove one of our main theorems.
\vskip 1mm
\noindent\textbf{Theorem~\ref{thm:GL-no-stable}.}
\emph{Every stable solution of~\eqref{eq:GL-eq} on $S^{n}$ for $n \geq 2$ is necessarily constant with absolute value $1$, regardless of the value of $\ep$.}
\begin{proof}
Let $u$ be a stable solution. Since $0$ is easily seen to be unstable, and since non-zero constant solutions must have absolute value $1$, we only have to prove that $u$ is constant. The idea of proof comes from the averaging method in~\cite{LS}. Namely, letting $\xi_{1}, \cdots, \xi_{n + 1}$ denote the standard basis for $\RR^{n + 1}$, we apply~\eqref{eq:2nd-innvar-conKill} to $\xi = \xi_{1}, \cdots, \xi_{n + 1}$ and sum up the result. Noting that 
\[
\sum_{i = 1}^{n + 1}f_{\xi_{i}}^{2} = 1  \text{ and }\sum_{i = 1}^{n + 1}|\langle \nabla u, X_{\xi_i} \rangle|^2 = |\nabla u|^2
\]
on $S^{n}$, and using Corollary~\ref{coro:inner-outer-coro}, we obtain 
\begin{equation}\label{eq:2nd-innvar-avg}
0 \leq \sum_{i = 1}^{n + 1} \delta^{2}E_{\ep}(u)(X_{\xi_i}, X_{\xi_{i}}) = -(n-2)\int_{M}|\nabla u|^{2}d\mu_{g}.
\end{equation}
When $n \geq 3$, this forces $\nabla u$ to vanish identically, and consequently $u$ is constant. When $n = 2$, Corollary~\ref{coro:inner-outer-coro} and~\eqref{eq:2nd-innvar-avg} imply that
\[
\delta^{2}E_{\ep}(u)(\nabla_{X_{\xi_i}}u, \nabla_{X_{\xi_i}}u) = \delta^{2}E_{\ep}(u)(X_{\xi_i}, X_{\xi_i}) = 0 \text{ for all }i.
\]
Since $u$ is a stable solution, this means that for all $i = 1, \cdots, n + 1$, the function $v_{i} = \langle \nabla u, X_{\xi_i} \rangle$ necessarily lies in the kernel of the operator associated with the bilinear form $\delta^{2}E_{\ep}(u)$. In other words, 
\begin{equation}\label{eq:kernel}
\Delta v_i  = \frac{|u|^2 - 1}{\ep^{2}}v_{i} + \frac{2(u \cdot v_{i})u}{\ep^{2}}.
\end{equation}
On the other hand, commuting derivatives and introducing Ricci curvature terms, we have
\begin{align*}
\Delta v_{i}&= \Delta \langle \nabla u, X_{\xi_i} \rangle = \langle \Delta \nabla u, X_{\xi_i} \rangle + 2\langle \nabla^{2}u, \nabla X_{\xi_i} \rangle + \langle \nabla u, \Delta X_{\xi_i} \rangle\\
&= \langle \nabla \Delta u + (n - 1)\nabla u, X_{\xi_i} \rangle + 2\langle \nabla^2 u, \nabla X_{\xi_i} \rangle + \langle \nabla u, \nabla \Delta f_{\xi_i} + (n - 1)\nabla f_{\xi_i} \rangle\\
&= \frac{|u|^2 - 1}{\ep^{2}}v_{i} + \frac{2(u \cdot v_{i})u}{\ep^{2}} + 2(n - 1)\langle \nabla u, X_{\xi_i} \rangle + 2\langle \nabla^2 u, \nabla X_{\xi_i} \rangle  + \langle \nabla u, \nabla \Delta f_{\xi_i} \rangle.
\end{align*}
Recalling~\eqref{eq:kernel} and that $\Delta f_{\xi_i} = -n f_{\xi_i}$, we deduce that 
\begin{equation}
(n - 2)\langle \nabla u, X_{\xi_i} \rangle + 2\langle \nabla^2 u, \nabla X_{\xi_i} \rangle = 0.
\end{equation}
Since we are in the case $n = 2$, the first term drops and we get that 
\begin{equation}
2\langle \nabla^2 u, \nabla X_{\xi_i} \rangle = 0 \text{ on $S^{n}$, for all }i.
\end{equation}
Using Proposition~\ref{prop:confKilling}(b), we arrive at
\begin{equation}
-f_{\xi_i}(x)\Delta u(x) = 0 \text{ for all }x \in S^{n},\ i = 1, \cdots, n + 1,
\end{equation}
which means that $\Delta u$ vanishes identically on $S^{n}$, and consequently $u$ is constant.
\end{proof}

Next we show how the comparison result, Proposition~\ref{prop:2nd-innvar-lim}, can be used to give index lower bounds for solutions to~\eqref{eq:GL-eq} on $S^{n}$ under appropriate assumptions. That is, we prove \\

\noindent\textbf{Theorem~\ref{thm:GL-index-bound}.}
\emph{Suppose $n \geq 3$. For all $ C> 0$, there exists $\ep_{0} > 0$ such that if $u$ is a solution on $S^{n}$ to~\eqref{eq:GL-eq} with $\ep < \ep_{0}$, and if
\begin{equation}\label{eq:E_0-bound}
C^{-1}|\log\ep| \leq E_{\ep}(u) \leq C|\log\ep|,
\end{equation}
then the Morse index of $u$ as a critical point of $E_{\ep}$ is at least $2$.}
\begin{proof}
We argue by contradiction. Negating the conclusion yields a $C > 0$ and a sequence $u_{\ep}$ of solutions to~\eqref{eq:GL-eq} with $\ep \to 0$ such that~\eqref{eq:E_0-bound} holds for all $\ep$, but each $u_{\ep}$ has index smaller than $2$. Applying Proposition~\ref{prop:2nd-innvar-lim} to $u_{\ep}$, we get a subsequence, which we do not relabel, and a stationary rectifiable $(n-2)$-varifold $V$ in $S^{n}$, such that~\eqref{eq:2nd-innvar-lim} holds. Moreover, the first inequality in~\eqref{eq:E_0-bound} implies that $V$ is non-trivial.

By Proposition~\ref{prop:confKilling}(b), we see that at every $x \in S^{n}$ where $T_{x}\Sigma$ exists, we have
\[
 \big(\langle \nabla_{\nu_1}X_\xi, \nu_2 \rangle + \langle \nabla_{\nu_2}X_\xi, \nu_1 \rangle\big)^{2} + \big(\langle \nabla_{\nu_1}X_\xi, \nu_1 \rangle - \langle \nabla_{\nu_2}X_\xi, \nu_2 \rangle\big)^{2} = 0 \text{ for all }\xi \in \RR^{n + 1}.
\]
Consequently the second term on the right of~\eqref{eq:2nd-innvar-lim} vanishes, and we are left with
\begin{equation*}
\lim_{\ep \to 0}\frac{1}{|\log\ep|}\delta^{2}E_{\ep}(u_{\ep})(X_\xi, X_\xi) = \delta^{2}V(X_{\xi}, X_{\xi}) \text{ for all }\xi \in \RR^{n + 1}.
\end{equation*}
Polarizing the quadratic forms involved and noting the linearity of $\xi \mapsto X_{\xi}$, we see that
\begin{equation}\label{eq:clean-lim}
\lim_{\ep \to 0}\frac{1}{|\log\ep|}\delta^{2}E_{\ep}(u_{\ep})(X_\xi, X_\eta) = \delta^{2}V(X_{\xi}, X_{\eta}) \text{ for all }\xi,\eta \in \RR^{n  + 1}.
\end{equation}
To continue, we note the following fact which is implicit in~\cite{LS}, namely that for any stationary varifold $V$ and any $\xi \in \RR^{n + 1}$ there holds
\begin{equation}\label{eq:varifold-unstable}
\delta^{2}V(X_\xi, X_\xi) = - (n-2) \int_{\Sigma} |(X_\xi)^{\perp}|^{2} d\|V\|.
\end{equation}
This can be proved, for instance, by substituting $X = X_\xi$ into~\eqref{eq:varifold-2nd-var} and combining the result with the fact that $\delta V(f_\xi X_\xi) = 0$ by stationarity. In any case, since $V \neq 0$, the identity~\eqref{eq:varifold-unstable} implies, as in~\cite[Section 5]{Sim}, that there exist $\xi, \eta \in \RR^{n + 1}$, linearly independent, such that $\delta^2 V$ is negative-definite when restricted to $\Span\{X_{\xi}, X_{\eta}\}$. But then we see by~\eqref{eq:clean-lim} that, for $\ep$ sufficiently small, the matrix
\[
\left(
\begin{array}{cc}
\delta^{2}E_{\ep}(u_{\ep})(X_{\xi}, X_{\xi}) & \delta^{2}E_{\ep}(u_{\ep})(X_{\xi}, X_{\eta})\\
\delta^{2}E_{\ep}(u_{\ep})(X_{\eta}, X_{\xi}) & \delta^{2}E_{\ep}(u_{\ep})(X_{\eta}, X_{\eta})
\end{array}
\right)
\]
is negative-definite. Recalling Corollary~\ref{coro:inner-outer-coro}, it follows that, for all sufficiently small $\ep$, 
\begin{equation}\label{eq:GL-outer-unstable}
\delta^{2}E_{\ep}(u_{\ep})(v, v) < 0 \text{ for all }v \in \Span_{\RR}\{\nabla_{X_{\xi}}u_{\ep}, \nabla_{X_{\eta}}u_{\ep}\}.
\end{equation}
To obtain a contradiction, it remains to show that $\nabla_{X_{\xi}}u_{\ep}$ and $\nabla_{X_{\eta}}u_{\ep}$ are linearly independent over $\RR$. To that end, suppose that $a\nabla_{X_{\xi}}u_{\ep} + b\nabla_{X_{\eta}}u_{\ep} \equiv 0$ for some $a, b \in \RR$. Then, letting $\varphi_t$ denote the flow generated by $X_{\widetilde{\xi}}$, where $\widetilde{\xi}:= a\xi + b\eta$, we have 
\[
u_{\ep}\circ \varphi_t = u_{\ep} \text{ for all }t.
\]
Recalling that, as $t\to \infty$, the conformal diffeomorphisms $\varphi_{t}: S^{n} \to S^{n}$ converge locally uniformly to a constant away from its antipodal point, we deduce that $u_{\ep}$ must be constant, so either $u_{\ep} \equiv 0$, in which case $E_{\ep}(u_{\ep}) = \frac{\mu_g(M)}{4\ep^2}$, or $|u_{\ep}| \equiv 1$, in which case $E_{\ep}(u_{\ep}) = 0$. But both possibilities are ruled out by~\eqref{eq:E_0-bound} when $\ep$ is small enough, and thus $\nabla_{X_{\xi}}u_{\ep}$ and $\nabla_{X_{\eta}}u_{\ep}$ must be linearly independent over $\RR$. Returning to~\eqref{eq:GL-outer-unstable}, we conclude that the index of $u_{\ep}$ is at least $2$ for small enough $\ep$, a contradiction.
\end{proof}

\section{Stable solutions on $\CC\PP^{n}$}
In this section we study solutions to~\eqref{eq:GL-eq} on $\CC\PP^{n}$ with the Fubini--Study metric $g = g_{FS}$ as introduced in Section 1.2. Again motivated by \cite{LS}, we make use of the real holomorphic vector fields on $\CC\PP^{n}$. Below we review some basic facts concerning these objects.

\subsection{Real holomorphic and Killing vector fields on $\CC\PP^{n}$}
Letting $J$ denote the complex structure on $\CC\PP^{n}$, recall that a vector field $V$ is real holomorphic if and only if $\cL_V J = 0$. Since $\CC\PP^n$ is  a compact K\"ahler manifold, any Killing vector field is real holomorphic~\cite{Moro}. Moreover, if $V$ is real holomorphic, then so is $JV$. Thus, letting $\mathcal{K}$ denote the set of Killing vector fields, it follows that the vector fields in
\[
J\mathcal{K} = \{JV\ |\ V \in \cK\}
\]
are real holomorphic vector fields. 

The vector fields in $J\cK$ are similar to vector fields $X_{\xi}$ on $S^{n}$ in that they are the gradients of eigenfunctions corresponding to the first non-zero eigenvalue of the Laplace operator on $\CC\PP^{n}$. These first eigenfunctions are given as follows~\cite{BGM}: For each matrix $w \in H_{n + 1}\setminus\{0\}$, where $H_{n + 1} = \{ w \in \CC^{(n + 1) \times (n + 1)}\ |\ w = w^{\ast},\ \tr w = 0 \}$, we define
\begin{equation}\label{eq:CPn-eigenfunction}
f_{w}(z_0, \cdots, z_n) = w_{ij}z_i \bar{z}_j.
\end{equation}
Viewing $\CC\PP^{n}$ as the quotient of $S^{2n + 1} \subset \CC^{n + 1}$ by the $U(1)$-action 
\[(z_0, \cdots, z_n) \mapsto (e^{i\theta}z_0, \cdots, e^{i\theta}z_n),
\]
then the restrictions of $f_{w}$ to $S^{2n + 1}$ are $U(1)$-invariant and descend to $\CC\PP^{n}$ to give the first eigenfunctions of $\Delta_{g}$, which we still call $f_{w}$, by slight abuse of notation. Recall that we normalized $g$ in such a way that $\Ric_g = \frac{n + 1}{2} g$ and 
\[\Delta_g f_w = -(n + 1)f_w \text{ for all }w\in H_{n + 1}.\]



Next we recall some facts about Killing vector fields on $\CC\PP^n$~\cite[Chapter XI]{KN}. Given $A \in \fsu(n + 1)$, the restrictions of $e^{tA}$ to $S^{2n + 1}$ descend to isometries of $\CC\PP^n$ acting by
\[
\varphi_t([z]) = [e^{tA}z]\text{ for }z \in \CC^{n + 1}.
\]
These produce a Killing field given by $W_A= \frac{d}{dt}\big|_{t = 0}\varphi_t$. The map $A \mapsto W_A$ from $\fsu(n + 1)$ to $\cK$ in fact is an isomorphism of Lie algebras, and induces an inner product on $\cK$ via
\begin{equation}\label{eq:Killing-form}
( W_A, W_B )_{\cK} = 2\tr(AB^{\ast}),\ A, B \in \fsu(n + 1).
\end{equation}
Moreover, we have
\begin{equation}\label{eq:adj-skew}
([W_A, W_B], W_C)_{\cK} = -(W_B, [W_A, W_C])_{\cK} \text{ for all }A, B, C \in \fsu(n + 1).
\end{equation}
More importantly, for any $x \in \CC\PP^n$, with respect to the inner product~\eqref{eq:Killing-form}, we have the orthogonal decomposition
\[
\cK = \ff_x \oplus \fp_x,
\]
where $\ff_x = \{V \in \cK\ |\ V\vert_{x} = 0\}$ and $\fp_x = \{V \in \cK\ |\ (\nabla V)\vert_{x} = 0\}$. The former generates isometries of $\CC\PP^{n}$ which fix $x$, while the latter is isometric to $T_{x}\CC\PP^{n}$ through the evaluation map $V \mapsto V|_{x}$. As such, each $\xi \in T_{x}\CC\PP^n$ can be uniquely extended to a Killing field $\widetilde{\xi}$ in $\fp_{x}$, such that
\begin{equation}\label{eq:identify-SUn}
(\widetilde{\xi}, \widetilde{\eta})_{\cK} = \langle \xi, \eta \rangle_{FS} \text{ for all }\xi, \eta \in T_{x}\CC\PP^{n}.
\end{equation}
When there is no danger of confusion, we will drop the tilde ( $\widetilde{\ }$ ) and rely on the context to distinguish between a tangent vector in $T_{x}\CC\PP^n$ and the Killing field in $\fp_{x}$ it determines.

Finally, by the fact that $\CC\PP^n$ is a symmetric space, or by direct computation, we have
\begin{equation}\label{eq:bracket-commutator}
[\fp_x, \fp_x] \subset \ff_x,\ [\ff_x, \ff_x] \subset \ff_x \text{ and }[\fp_x, \ff_x] \subset \fp_x.
\end{equation}
For instance, the last fact can be seen as follows: take $X \in \fp_x$ and $V \in \ff_x$. Then, for all $W \in T_{x}\CC\PP^n$, we have, at the point $x$,
\begin{align*}
\nabla_W \big( [X, V] \big)&= \nabla_W \big( \nabla_X V  - \nabla_V X\big)\\
&= \nabla^2_{W, X}V + \nabla_{\nabla_W X}V - \nabla^{2}_{W, V}X - \nabla_{\nabla_W V}X.
\end{align*}
Since $X \in \fp_x$, the second and fourth terms in the last line vanish at $x$, whereas the first and third terms can be rewritten using the identity 
\begin{equation}\label{eq:Killing-curvature}
\nabla^{2}_{Y, Z}X = R_{Y, X}Z \text{ for $X$ Killing and $Y, Z$ arbitrary.}
\end{equation}  
Thus, because $V|_x = 0$, we have
\[
\nabla_W \big( [X, V] \big) = R_{W, V}X - R_{W, X} V = 0 \text{ at }x,
\]
for any $W \in T_{x}\CC\PP^n$. That is, the Killing field $[X, V]$ belongs to $\fp_x$.

\subsection{Stable solutions on $\CC\PP^{n}$}
We are now ready to state the key computation.
\begin{prop}\label{prop:GL-CPn}
Let $u: \CC\PP^{n} \to \CC$ be a solution to~\eqref{eq:GL-eq}, and let $V_1, \cdots, V_{q}$ be any orthonormal basis of $\cK$. Then
\begin{equation}\label{eq:GL-CPn}
\sum_{i = 1}^{q}\delta^{2}E_{\ep}(u)(JV_{i}, JV_{i}) = 0.
\end{equation}
\end{prop}
We first prove the following lemma which will be used repeatedly in the proof of Proposition~\ref{prop:GL-CPn}. The lemma actually follows from the calculations in~\cite[p.447]{LS}, but we include the proof for completeness.
\begin{lemm}\label{lemm:GL-CPn-prelim}
Fix $x \in \CC\PP^{n}$ and let $V_{1}, \cdots, V_{2n}$ and $V_{2n + 1}, \cdots, V_{q}$ be, respectively, orthonormal bases for $\fp_{x}$ and $\ff_{x}$. 
Then for any $X, Y, W, Z \in T_{x}\CC\PP^{n} \simeq \fp_x$, we have
\begin{equation}\label{eq:GL-CPn-prelim}
\sum_{k = 1}^{q}\langle \nabla_{X}JV_{k}, Y \rangle \langle \nabla_{W}JV_{k}, Z \rangle = \langle R_{X, JY}JW, Z \rangle.
\end{equation}
\end{lemm}
\begin{proof}
Throughout this proof we use the same notation for vectors in $T_{x}\CC\PP^{n}$ and the Killing fields in $\fp_{x}$ they induce. To begin, note that since $J$ is parallel and $(\nabla V_{k})\vert_{x} = 0$ for $k = 1, \cdots, 2n$, we have 
\begin{align}
\sum_{k = 1}^{q}\langle \nabla_{X}JV_{k}, Y \rangle \langle \nabla_{W}JV_{k}, Z \rangle  &= \sum_{k = 1}^{q}\langle \nabla_{X}V_{k}, JY \rangle \langle \nabla_{W}V_{k}, JZ \rangle \\
&= \sum_{k = 2n + 1}^{q} \langle \nabla_{X}V_{k}, JY \rangle \langle \nabla_{W}V_{k}, JZ \rangle\\
&= \sum_{k = 2n + 1}^{q} \langle [X, V_{k}], JY \rangle \langle [W, V_{k}], JZ \rangle,
\end{align}
where in the last line we used the fact that since $V_{k} \in \ff_{x}$ when $k > 2n$, we have
\[
\nabla_{\cdot} V_{k} =  [\cdot, V_{k}] \text{ at }x.
\]
To continue, note that since $X, W \in \fp_{x}$, and $V_{k} \in \ff_{x}$, by~\eqref{eq:bracket-commutator} above we have $[X, V_{k}], [W, V_{k}] \in \fp_{x}$, and hence by~\eqref{eq:identify-SUn}, we have
\begin{align}
\sum_{k = 2n + 1}^{q} \langle [X, V_{k}], JY \rangle \langle [W, V_{k}], JZ \rangle&= \sum_{k = 2n + 1}^{q} \big([X, V_{k}], \widetilde{JY})_{\cK}([W, V_{k}], \widetilde{JZ} \big)_{\cK}\\
&=  \sum_{k = 2n + 1}^{q} \big(V_{k}, [X, \widetilde{JY}] )_{\cK} (V_{k}, [W, \widetilde{JZ}] \big)_{\cK},
\end{align}
where we used~\eqref{eq:adj-skew} in getting the last line. Now since $X, W, \widetilde{JY}, \widetilde{JZ} \in \fp_x$, we have $[X, \widetilde{JY}], [W, \widetilde{JZ}] \in \ff_{x}$ and, as $\ff_{x}$ and $\fp_{x}$ are orthogonal, the last line is no other than
\begin{align*}
\big( [X, JY], [W, JZ] \big)_{\cK}&= - \big( [W, [X, \widetilde{JY}]], \widetilde{JZ}\big)_{\cK}\\
&= - \big\langle [W, [X, \widetilde{JY}]]\big|_{x}, \widetilde{JZ}\big|_{x} \big\rangle\ (\text{by }\eqref{eq:identify-SUn}, \text{ since }[W, [X, \widetilde{JY}]] \in \fp_x)\\
&= -\langle R_{X, JY}W, JZ \rangle\ (\text{since }[Z, [X, Y]] = R_{X, Y}Z \text{ for }X, Y, Z \in \fp_{x})\\
&=\langle R_{X, JY}JW, Z \rangle,
\end{align*}
where the identity mentioned in the third line follows from~\eqref{eq:Killing-curvature} and the first Bianchi identity, and  is applied to $X, \widetilde{JY}, W$ in place of $X, Y, Z$, respectively. The proof is complete.
\end{proof}
\begin{proof}[Proof of Proposition~\ref{prop:GL-CPn}]
We carry out the (rather lengthy) computation using the formula~\eqref{eq:2nd-innervar}. To simplify notation we let
\begin{align*}
Q_{1}(X) &= (\Div X)^{2} - \Ric(X, X) - \langle \nabla_{e_{i}}X, e_{j} \rangle \langle \nabla_{e_{j}}X, e_{i} \rangle + \Div \nabla_X X\\
Q_{2}(X)&= 2\langle \nabla_{\nabla u}X, \nabla u \rangle  \Div X  - R(\nabla u, X, X, \nabla u) + |\nabla_{\nabla u}X|^{2} + \langle \nabla_{\nabla u}\nabla_{X} X, \nabla u \rangle\\
Q_{3}(X) &= |\cL_{X}g\ \llcorner \nabla u |^{2}.
\end{align*}
Clearly it suffices to prove that at each point $x \in \CC\PP^{n}$, there is \textit{some} orthonormal basis $V_{1}, \cdots, V_{q}$ of $\cK$ such that 
\begin{equation}\label{eq:GL-CPn-pointwise}
\sum_{k= 1}^{q}Q_{1}(JV_{k}) = 0;\   \sum_{k = 1}^{q} \big(-Q_{2}(JV_{k}) + Q_{3}(JV_{k})\big) = 0.
\end{equation}
Thus let us fix an arbitrary $x \in \CC\PP^{n}$ and assume that $V_{1}, \cdots, V_{2n}$ is an orthonormal basis for $\fp_{x} \simeq T_{x}\CC\PP^{n}$, and $V_{2n + 1}, \cdots, V_{q}$ is an orthonormal basis for $\ff_{x}$. Moreover, we let
\[
e_{i} = V_{i}\vert_{x} \text{ for }i = 1, \cdots, 2n,
\]
which form an orthonormal basis of $T_{x}\CC\PP^{n}$.

We begin by analyzing $\sum_{k = 1}^{q}Q_{1}(JV_{k})$. First note that  
\begin{align*}
\sum_{k = 1}^{q} (\Div JV_{k})^{2} &= \sum_{k = 1}^{q}\sum_{i, j = 1}^{2n} \langle \nabla_{e_{i}}JV_{k}, e_{i} \rangle \langle \nabla_{e_{j}}JV_{k}, e_{j} \rangle.
\end{align*}
Combining this with Lemma~\ref{lemm:GL-CPn-prelim}, we get
\begin{equation}\label{eq:div-2}
\sum_{k = 1}^{q} (\Div JV_{k})^{2} = \sum_{i, j = 1}^{2n}\langle R_{e_i, Je_i}Je_j, e_j \rangle.
\end{equation}
Similarly, for the third term in the definition of $Q_{1}$ we have
\begin{equation}\label{eq:contraction-2}
\sum_{k = 1}^{q}\sum_{i, j = 1}^{2n} \langle \nabla_{e_{i}}JV_{k}, e_{j} \rangle\langle \nabla_{e_{j}}JV_{k}, e_{i} \rangle = \sum_{i, j = 1}^{2n}\langle R_{e_{i}, Je_{j}}Je_{j}, e_{i} \rangle.
\end{equation}
For the Ricci term in $Q_{1}$, we simply have
\begin{equation}
\sum_{j = 1}^{q}\Ric(JV_{j}, JV_{j})= \sum_{j = 1}^{q}\sum_{i = 1}^{2n}\langle R_{e_{i}, JV_{j}}JV_{j}, e_{i} \rangle = \sum_{i, j = 1}^{2n}\langle R_{e_{i}, Je_{j}}Je_{j}, e_{i} \rangle,
\end{equation}
where the second equality holds because $JV_{j} = 0$ at $x$ for all $j = 2n + 1, \cdots, q$.
Finally, for the last term in the definition of $Q_{1}$, we have
\begin{align*}
\sum_{k = 1}^{q}\Div \big(\nabla_{JV_{k}}JV_{k}\big) &= \sum_{k = 1}^{q}\sum_{i = 1}^{2n}\langle \nabla_{e_{i}}\nabla_{JV_{k}}JV_{k}, e_{i} \rangle\\
&= -\sum_{k = 1}^{q}\sum_{i = 1}^{2n}\langle \nabla_{e_{i}}\nabla_{JV_{k}}V_{k}, Je_{i} \rangle\\
&= -\sum_{k = 1}^{q}\sum_{i = 1}^{2n}\langle \nabla^{2}_{e_{i}, JV_{k}}V_{k}, Je_{i} \rangle - \sum_{k = 1}^{q}\sum_{i = 1}^{2n}\langle \nabla_{\nabla_{e_{i}}JV_{k}}V_{k}, Je_{i} \rangle.
\end{align*}
To continue, we apply~\eqref{eq:Killing-curvature} to the first term and transform the last line to
\begin{align*}
&-\sum_{k = 1}^{q}\sum_{i = 1}^{2n}\langle R_{e_{i}, V_{k}}JV_{k}, Je_{i} \rangle - \sum_{k = 1}^{q}\sum_{i, j = 1}^{2n}\langle \nabla_{e_{i}}JV_{k}, e_{j} \rangle \langle \nabla_{e_{j}}V_{k}, Je_{i} \rangle\\
&= -\sum_{i, j = 1}^{2n}\langle R_{e_{i}, e_{j}}e_{j}, e_{i} \rangle +\sum_{k = 1}^{q}\sum_{i, j = 1}^{2n} \langle \nabla_{e_{i}}JV_{k}, e_{j} \rangle\langle \nabla_{e_{j}}JV_{k}, e_{i} \rangle\\
&= -\sum_{i, j = 1}^{2n}\langle R_{e_{i}, e_{j}}e_{j}, e_{i} \rangle + \sum_{i, j = 1}^{2n}\langle R_{e_{i}, Je_{j}}Je_{j}, e_{i} \rangle,
\end{align*}
where we used Lemma~\ref{lemm:GL-CPn-prelim} in getting the last line. To sum up, we arrive at
\begin{equation}
\sum_{k = 1}^{q}Q_{1}(JV_{k}) = \sum_{i, j = 1}^{2n}\langle R_{e_i, Je_i}Je_j, e_j \rangle -\sum_{i, j = 1}^{2n}\langle R_{e_{i}, e_{j}}e_{j}, e_{i} \rangle - \sum_{i, j = 1}^{2n}\langle R_{e_{i}, Je_{j}}Je_{j}, e_{i} \rangle.
\end{equation}
By the symmetries of the curvature tensor and the first Bianchi identity, we can combine the first and third terms above to get
\begin{align*}
\sum_{i, j = 1}^{2n}\langle R_{e_i, Je_i}Je_j, e_j \rangle - \sum_{i, j = 1}^{2n}\langle R_{e_{i}, Je_{j}}Je_{j}, e_{i} \rangle &= \sum_{i, j = 1}^{2n}\langle R_{e_i, Je_i}Je_j, e_j \rangle + \sum_{i, j = 1}^{2n}\langle R_{Je_{j}, e_i}Je_{i}, e_{j} \rangle\\
&= -\sum_{i, j = 1}^{2n} \langle R_{Je_{i}, Je_{j}}e_{i}, e_{j}\rangle\\
&= \sum_{i, j = 1}^{2n} \langle R_{e_i, e_j}e_j, e_i \rangle.
\end{align*}
Therefore we conclude that 
\begin{equation}\label{eq:Q1-avg}
\sum_{k = 1}^{q}Q_{1}(JV_{k})  = 0.
\end{equation}

Moving on to $Q_{2}$, we have, again by Lemma~\ref{lemm:GL-CPn-prelim},
\begin{align*}
\sum_{k = 1}^{q} \langle \nabla_{\nabla u}JV_k, \nabla u \rangle \Div (JV_{k})&= \sum_{k = 1}^{q}\sum_{i = 1}^{2n} \langle \nabla_{\nabla u}JV_k, \nabla u \rangle \langle \nabla_{e_{i}}JV_{k}, e_{i} \rangle = \sum_{i = 1}^{2n}\langle R_{\nabla u, J\nabla u}Je_{i}, e_{i} \rangle.
\end{align*}
To continue, note that 
\begin{align*}
\sum_{i= 1}^{q}R(\nabla u, JV_{i}, JV_{i}, \nabla u)&= \sum_{i = 1}^{2n}\langle R_{\nabla u, Je_{i}}Je_{i}, \nabla u \rangle.
\end{align*}
For the third term in the definition of $Q_{2}$, we have
\begin{align*}
\sum_{k = 1}^{q}|\nabla_{\nabla u}JV_{k}|^{2}&= \sum_{k = 1}^{q}\sum_{i = 1}^{2n}\langle \nabla_{\nabla u}JV_{k}, e_{i} \rangle \langle \nabla_{\nabla u}JV_{k}, e_{i} \rangle\\
&= \sum_{i = 1}^{2n}\langle  R_{\nabla u, Je_{i}} J\nabla u, e_{i} \rangle,
\end{align*}
where the last line follows by Lemma~\ref{lemm:GL-CPn-prelim}. Finally, we compute
\begin{align*}
\sum_{k = 1}^{q}\langle \nabla_{\nabla u}\nabla_{JV_{k}}JV_{k}, \nabla u \rangle&= -\sum_{k = 1}^{q}\langle \nabla_{\nabla u}\nabla_{JV_{k}}V_{k}, J\nabla u \rangle\\
&= -\sum_{k = 1}^{q}\langle \nabla^{2}_{\nabla u, JV_{k}}V_{k}, J\nabla u \rangle - \sum_{k = 1}^{q}\langle \nabla_{\nabla_{\nabla u}JV_{k}}V_{k}, J\nabla u \rangle\\
&= -\sum_{k = 1}^{q}\langle R_{\nabla u, V_{k}}JV_{k}, J \nabla u \rangle - \sum_{k = 1}^{q}\sum_{i = 1}^{2n}\langle \nabla_{\nabla u}JV_{k}, e_{i} \rangle \langle \nabla_{e_{i}}V_{k}, J\nabla u \rangle\\
& = -\sum_{k  =1}^{q}\langle R_{\nabla u, e_{k}}e_{k}, \nabla u \rangle + \sum_{k = 1}^{q}\sum_{i = 1}^{2n}\langle \nabla_{\nabla u}JV_{k}, e_{i} \rangle \langle \nabla_{e_{i}}JV_{k}, \nabla u \rangle\\
& = -\sum_{i  =1}^{2n}\langle R_{\nabla u, e_{i}}e_{i}, \nabla u \rangle + \sum_{i = 1}^{2n}\langle R_{\nabla u, Je_{i}}Je_{i}, \nabla u \rangle,
\end{align*}
where in getting from the second line to the third,~\eqref{eq:Killing-curvature} is again used. To sum up, we get
\begin{align}
\nonumber \sum_{k = 1}^{q}Q_{2}(JV_{k}) =& 2\sum_{i = 1}^{2n}\langle R_{\nabla u, J\nabla u}Je_{i}, e_{i} \rangle- \sum_{i = 1}^{2n}\langle R_{\nabla u, Je_{i}}Je_{i}, \nabla u \rangle\\
\nonumber & + \sum_{i = 1}^{2n}\langle  R_{\nabla u, Je_{i}} J\nabla u, e_{i} \rangle -\sum_{i  =1}^{2n}\langle R_{\nabla u, e_{i}}e_{i}, \nabla u \rangle + \sum_{i = 1}^{2n}\langle R_{\nabla u, Je_{i}}Je_{i}, \nabla u \rangle\\
\label{eq:Q2-avg} =& 2\sum_{i = 1}^{2n}\langle R_{\nabla u, J\nabla u}Je_{i}, e_{i} \rangle + \sum_{i = 1}^{2n}\langle  R_{\nabla u, Je_{i}} J\nabla u, e_{i} \rangle -\sum_{i  =1}^{2n}\langle R_{\nabla u, e_{i}}e_{i}, \nabla u \rangle.
\end{align}

For $Q_{3}$, recalling that $JV_{k}$ are gradient vector fields, we have
\[
(\cL_{JV_{k}}g)_{\nabla u, e_{i}} = 2\langle \nabla_{e_{i}}JV_{k}, \nabla u \rangle = 2\langle \nabla_{\nabla u}JV_{k}, e_{i} \rangle.
\]
Thus, we compute
\begin{align*}
\sum_{k = 1}^{q}Q_{3}(JV_{k}) &= \sum_{k = 1}^{q}\sum_{i = 1}^{2n} |(\cL_{JV_{k}}g)_{\nabla u, e_{i}}|^{2}\\
&=4\sum_{k = 1}^{q}\sum_{i = 1}^{2n} \langle \nabla_{e_{i}}JV_{k}, \nabla u \rangle \langle \nabla_{\nabla u}JV_{k}, e_{i} \rangle\\
&= 4\sum_{i = 1}^{2n}\langle R_{e_{i}, J\nabla u}J\nabla u, e_{i} \rangle\\
&= -4\sum_{i = 1}^{2n}\langle R_{Je_{i}, \nabla u}J\nabla u, e_{i} \rangle.
\end{align*}
Combining this with~\eqref{eq:Q2-avg}, we get
\begin{align}
\nonumber \sum_{k = 1}^{q}\big(-Q_{2}(JV_{k}) + Q_{3}(JV_{k})\big)=& -2\sum_{i = 1}^{2n}\langle R_{\nabla u, J\nabla u}Je_{i}, e_{i} \rangle - \sum_{i = 1}^{2n}\langle  R_{\nabla u, Je_{i}} J\nabla u, e_{i} \rangle +\sum_{i  =1}^{2n}\langle R_{\nabla u, e_{i}}e_{i}, \nabla u \rangle\\
\nonumber &-4\sum_{i = 1}^{2n}\langle R_{Je_{i}, \nabla u}J\nabla u, e_{i} \rangle\\
\nonumber =& -2\sum_{i = 1}^{2n}\langle R_{\nabla u, J\nabla u}Je_{i}, e_{i} \rangle + 3\sum_{i = 1}^{2n}\langle  R_{\nabla u, Je_{i}} J\nabla u, e_{i} \rangle +\sum_{i  =1}^{2n}\langle R_{\nabla u, e_{i}}e_{i}, \nabla u \rangle\\
\nonumber =& -2\sum_{i = 1}^{2n}\langle R_{\nabla u, J\nabla u}Je_{i}, e_{i} \rangle -2 \sum_{i = 1}^{2n}\langle  R_{Je_{i}, \nabla u} J\nabla u, e_{i} \rangle\\
\label{eq:Q2-Q3-sum}& +\sum_{i = 1}^{2n}\langle  R_{\nabla u, Je_{i}} J\nabla u, e_{i} \rangle+\sum_{i  =1}^{2n}\langle R_{\nabla u, e_{i}}e_{i}, \nabla u \rangle.
\end{align}
By the Bianchi identity, we have
\begin{align*}
& -2\sum_{i = 1}^{2n}\langle R_{\nabla u, J\nabla u}Je_{i}, e_{i} \rangle -2 \sum_{i = 1}^{2n}\langle  R_{ Je_{i}, \nabla u} J\nabla u, e_{i} \rangle\\ 
=&\ 2\sum_{i = 1}^{2n}\langle R_{J\nabla u, Je_{i}}\nabla u, e_{i} \rangle = -2\sum_{i = 1}^{2n}\langle R_{\nabla u, e_{i}}e_{i}, \nabla u \rangle.
\end{align*}
Putting this back to~\eqref{eq:Q2-Q3-sum}, we get
\begin{align*}
 \sum_{k = 1}^{q}\big(-Q_{2}(JV_{k}) + Q_{3}(JV_{k})\big)&= \sum_{i = 1}^{2n}\big( \langle R_{\nabla u, Je_{i}}J\nabla u ,e_{i} \rangle - \langle R_{\nabla u, e_{i}}e_{i}, \nabla u \rangle\big)\\
 &= \sum_{i = 1}^{2n}\big( \langle R_{\nabla u, Je_{i}}Je_{i} ,\nabla u \rangle - \langle R_{\nabla u, e_{i}}e_{i}, \nabla u \rangle\big)\\
 & = \Ric(\nabla u, \nabla u) - \Ric(\nabla u, \nabla u) = 0.
\end{align*}
Hence, recalling~\eqref{eq:Q1-avg}, we see that~\eqref{eq:GL-CPn-pointwise} is proved, and we are done.
\end{proof}
We can now imitate the argument used to handle the $n = 2$ case in the proof of Theorem~\ref{thm:GL-no-stable} to establish our main result about solutions on $\CC\PP^{n}$.
\vskip 1mm
\noindent\textbf{Theorem~\ref{thm:GL-no-stable-CPn}.}
\emph{For $n \geq 1$, every stable solutions to~\eqref{eq:GL-eq} on $\CC\PP^{n}$ is necessarily constant with absolute value $1$, regardless of the value of $\ep$.}
\begin{proof}
Suppose $u:\CC\PP^{n} \to \CC$ is a stable solution to~\eqref{eq:GL-eq}. As in the proof of Theorem~\ref{thm:GL-no-stable}, we only need to prove that $u$ is constant. By stability and Corollary~\ref{coro:inner-outer-coro}, we have
\[
\delta^{2}E_{\ep}(u)(\nabla_{JV}u, \nabla_{JV}u) = \delta^2 E_{\ep}(u)(JV, JV) \geq 0 \text{ for all }V \in \cK.
\]
Combining this with Proposition~\ref{prop:GL-CPn}, we see that 
\[
\delta^{2}E_{\ep}(u)(\nabla_{JV}u, \nabla_{JV}u) = \delta^{2}E_{\ep}(u)(JV, JV) = 0 \text{ for all }V \in \cK.
\]
Consequently, for all $V \in \cK$, the function $v := \nabla_{JV}u$ lies in the kernel of the bilinear form $\delta^{2}E_{\ep}(u)$. That is,
\begin{equation}\label{eq:GL-CPn-Jacobi}
\Delta v = \frac{|u|^{2} - 1}{\ep^{2}}v + \frac{2(u \cdot v)u}{\ep^{2}}. 
\end{equation}
On the other hand, recalling that $JV = \nabla f$ for some eigenfunction $\Delta f = -(n + 1)f$, and that $\Ric_{g} = \frac{n + 1}{2}g$, we compute $\Delta v$ directly:
\begin{align*}
\Delta v &= \Delta \langle \nabla u, \nabla f \rangle = \langle \Delta \nabla u, \nabla f \rangle + 2\langle \nabla^{2}f, \nabla^{2}u \rangle + \langle \nabla u, \Delta \nabla f \rangle\\
&= \langle \nabla \Delta u , \nabla f \rangle + \Ric(\nabla u, \nabla f) + 2\langle \nabla^{2}f, \nabla^{2}u \rangle + \langle \nabla u, \nabla \Delta f \rangle + \Ric(\nabla u, \nabla f)\\
&= \frac{|u|^{2} - 1}{\ep^{2}}v + \frac{2(u \cdot v)u}{\ep^{2}} + 2\Ric(\nabla u, \nabla f) - (n + 1)\langle \nabla u, \nabla f \rangle + 2\langle \nabla^{2}f, \nabla^{2}u \rangle\\
&= \frac{|u|^{2} - 1}{\ep^{2}}v + \frac{2(u \cdot v)u}{\ep^{2}} + 2\langle \nabla^{2}f, \nabla^{2}u \rangle.
\end{align*}
Combining this with~\eqref{eq:GL-CPn-Jacobi} and recalling that $V$ is any Killing vector field, we see that 
\begin{equation}\label{eq:Hessian-relation}
\langle \nabla^{2}f, \nabla^{2}u \rangle\vert_{x} = 0 \text{ for all }x \in \CC\PP^{n} \text{ and }f \text{ such that }\Delta f = -(n + 1)f.
\end{equation}
We claim that~\eqref{eq:Hessian-relation} implies that $\Delta u\vert_{x} = 0$ for all $x$ on $\CC\PP^{n}$. By homogeneity, it suffices to prove this at $x  = [1,0, \cdots,0]$. For this we introduce local coordinates $(z_1, \cdots, z_n) \mapsto [\frac{1}{\sqrt{(1 + |z|^{2})}},\frac{z_1}{\sqrt{(1 + |z|^{2})}},\cdots,\frac{z_n}{{\sqrt{(1 + |z|^{2})}}}]$ and write
\[
X_j = \paop{x^j},\ Y_j = \paop{y^{j}}
\]
\[
Z_j = \frac{1}{2}(X_j - \sqrt{-1}Y_j),\ \overline{Z}_j = \frac{1}{2}(X_j + \sqrt{-1}Y_j).
\]
In~\eqref{eq:CPn-eigenfunction}, we choose
\begin{equation}
w = \left(
\begin{array}{cc}
1 & 0\\
0 & -\frac{\delta_{ij}}{n}
\end{array}
\right).
\end{equation}
Writing $f$ for $f_w$ defined as in~\eqref{eq:CPn-eigenfunction}, in terms of coordinates we have
\[
f_{w}(z_{1}, \cdots, z_n) = \frac{1 - \frac{|z|^{2}}{n}}{1 + |z|^{2}}.
\]
By a direct computation, at the point $z = 0$ we have
\begin{align}\label{eq:2,0-derivative}
\nabla^{2}_{Z_k, Z_l}f = \frac{\partial^2 f}{\partial z^k \partial z^l} = 0 \Longrightarrow \nabla^{2}_{X_k, X_l}f = \nabla^{2}_{Y_k, Y_l}f \text{ and }\nabla^{2}_{X_k, Y_l}f = -\nabla^{2}_{Y_k, X_l}f.
\end{align}
Moreover, at $z = 0$ we also have
\begin{align*}
&\nabla^2_{Z_k, \overline{Z}_l}f = \frac{\partial^2 f}{\partial z^k \partial \overline{z}^l} = -\frac{n+1}{n}\delta_{kl}.
\end{align*}
Expressing $Z_{k}, \overline{Z}_l$ in terms of $X$ and $Y$, and combining with~\eqref{eq:2,0-derivative}, we see that 
\begin{align*}
\nabla^{2}_{Y_k, Y_l}f &= \nabla^{2}_{X_k, X_l}f = -2\frac{n+1}{n}\delta_{kl},\\
\nabla^{2}_{X_k, Y_l}f &= -\nabla^{2}_{Y_k, X_l}f = 0.
\end{align*}
Moreover, by our normalization of the Fubini--Study metric (see~\eqref{eq:FS-local} in particular), at $z = 0$ we have
\[
|X_{k}| = |Y_{k}| = 2 \text{ and }\langle X_{k}, Y_{l} \rangle = 0, \text{ for all }k, l.
\]
Therefore at $z = 0$, which corresponds to $x = [1, 0, \cdots, 0]$ on $\CC\PP^{n}$, we simply have
\[
\nabla^{2}f = -\frac{n + 1}{2n}g,
\]
in which case the vanishing condition~\eqref{eq:Hessian-relation} implies $(\Delta u)\vert_{x} = 0$. Repeating a similar argument at other points, we conclude that $\Delta u $ vanishes identically on $\CC\PP^{n}$, but then $u$ is necessarily constant, and the proof is complete.
\end{proof}
\section{First and second variations of $F_{\ep}$}
\subsection{Preliminaries}
In this section we switch gears and consider critical points of $F_{\ep}$. Let $L$ be a complex line bundle over a closed Riemannian $n$-manifold $(M, g)$, and suppose $L$ is equipped with a Hermitian metric $\langle\cdot, \cdot\rangle_{L}$. Below, unless otherwise stated, all the connections on $L$ we consider are metric connections. 

Fixing a smooth connection $D_{0}$ on $L$, any other connection $D$ can be written as 
\[D = D_0 - \sqrt{-1}a\]
for some real $1$-form $a$, and we say that $D$ is a $W^{1, 2}$-connection if the $1$-form $a$ is of class $W^{1, 2}$. We use $\mathcal{C}$ to denote the set of pairs $(u, D)$ where $u$ is a section of class $W^{1, 2} \cap L^{\infty}$ and $D$ is a $W^{1, 2}$-connection. Given $(u, D), (\widetilde{u}, \widetilde{D}) \in \mathcal{C}$, we say that they are gauge equivalent if there exists $\theta \in W^{2, 2}(M; \RR)$ such that
\begin{equation}\label{eq:gauge-eq-defi}
(\widetilde{u}, \widetilde{D}) = (e^{\sqrt{-1}\theta}u, D - \sqrt{-1}d\theta) \text{ on }M, 
\end{equation}
in which case we have $F_{\ep}(u, D) = F_{\ep}(\widetilde{u}, \widetilde{D})$, and that
\begin{equation}\label{eq:YMH-gauge-rel}
F_{\widetilde{D}} = F_{D},\ \widetilde{D}\widetilde{u} = e^{\sqrt{-1}\theta}Du.
\end{equation}
Of course the notion of gauge equivalence can be localized to any subset $\Omega \subset M$ by requiring instead that $\theta \in W^{2, 2}(\Omega; \RR)$ and that~\eqref{eq:gauge-eq-defi} holds on $\Omega$.

Suppose $(u, D) \in \mathcal{C}$. For a smooth section $v$ and a smooth $1$-form $a$, we define the first and second outer variations of $F_{\ep}$ at $(u, D)$ along $(v, a)$ to be
\begin{align}\label{eq:YMH-1st-outer}
&\nonumber\delta F_{\ep}(u, D)((v, a)) = \frac{d}{dt}F_{\ep}(u + tv, D - t\sqrt{-1}a)\big|_{t=0}\\ 
=& \int_{M} 2\ep^{2}\langle F_{D}, da \rangle + 2\re\langle Du, Dv - \sqrt{-1}au \rangle+ \frac{|u|^{2} - 1}{\ep^{2}}\re \langle u, v\rangle_{L}d\mu_g.
\end{align}
\begin{align}
\label{eq:YMH-2nd-outer}&\delta^{2} F_{\ep}(u, D)((v, a))=  \frac{d^2}{dt^2}F_{\ep}(u + tv, D - t\sqrt{-1}a)\big|_{t=0}=\\
\nonumber=& \int_{M} 2\ep^{2}|da|^{2} + 2|Dv - \sqrt{-1}au|^{2} -4\re\langle \sqrt{-1}av, Du \rangle + \frac{|u|^{2} - 1}{\ep^{2}}|v|^{2} + \frac{2\big(\re\langle u, v\rangle_{L}\big)^{2}}{\ep^{2}}d\mu_g.
\end{align}
Thus $(u, D)$ is a weak solution to~\eqref{eq:YMH-eq} if and only if $\delta F_{\ep}(u, D) (v, a)= 0$ for all smooth variations $(v, a)$. We say that a weak solution in $\cC$ to~\eqref{eq:YMH-eq} is stable if 
\begin{equation}\label{eq:YMH-outer-stable}
\delta^{2} F_{\ep}(u, D)((v, a)) \geq 0 \text{ for all smooth variations }(v, a).
\end{equation}

As in Section 2, we want to define variations of $F_{\ep}$ with respect to deformations of the domain. However, the computations require some regularity to go through, while a weak solution may not be globally gauge equivalent to a smooth one. Hence, we work on contractible domains $\Omega \subset M$ with sufficiently smooth boundary. In Section 6 we patch the local computations together in the case $M = S^{n}$.

Since $\Omega$ is contractible, we can fix a unitary trivialization $L\vert_{\Omega} \simeq \Omega \times \CC$, so that sections of $L\vert_{\Omega}$ are identified with complex-valued functions, their covariant derivatives with complex-valued $1$-forms, and the bundle metric simply becomes $\langle u, v \rangle_{L} = u\bar{v}$, so that 
\[
\re \langle u, v\rangle_{L} = u \cdot v,
\]
where $\cdot$ denotes the inner product on $\CC \simeq \RR^{2}$. Moreover, we can express a connection $D$ as $d - \sqrt{-1}A$ with $A$ being a real $1$-form on $\Omega$, in which case $F_{D} = dA$.

Now suppose $(u, D) \in \mathcal{C}$ is a weak solution to~\eqref{eq:YMH-eq} on $M$. Restricting to $\Omega$, we may write $D = d - \sqrt{-1}A$ as above. By the arguments in~\cite[Appendix A]{PiSt}, if we let $\theta \in W^{2, 2}(\Omega; \RR)$ be the unique solution to the following Neumann problem which integrates to zero on $\Omega$:
\begin{equation}\label{eq:gauge-fixing}
\left\{
\begin{array}{cl}
\Delta \theta &= d^{\ast}A \text{ in }\Omega,\\
\partial_{\nu}\theta & = - A_{\nu}  \text{ on }\partial\Omega,
\end{array}
\right.
\end{equation}
and define
\begin{equation}\label{eq:smooth-rep}
(\widetilde{u}, \widetilde{D}) = (e^{\sqrt{-1}\theta}u, D - \sqrt{-1}d\theta),
\end{equation}
then $(\widetilde{u}, \widetilde{D})$ is a smooth solution to~\eqref{eq:YMH-eq} on $\Omega$. Consequently, since $F_{D} = F_{\widetilde{D}}$ and $\Omega \subset M$ is any contractible domain, the form $F_D$ is in fact smooth on all of $M$. Similarly, $|D u|^{2}, |u|^{2}$ and $|D_{X}u|^2$ for any smooth vector field are all smooth over $M$. Note that the maximum principle applied to $|u|^2 - 1$ implies that $|u| \leq 1$ on $M$.

In principle, any gauge invariant expression manufactured out of $(u, D)$ patches together to define smooth objects on all of $M$. For later use, we note the following two examples.
\begin{lemm}\label{lemm:patching}
Let $(u, D) \in \cC$ be a weak solution to~\eqref{eq:YMH-eq} and let $X$ be a vector field on $M$. Then we have 
\begin{enumerate}
\item[(a)] For all $k \geq 1$ there exists a smooth, complex-valued $1$-form $\alpha$ on $M$ that restricts to 
\[
\langle \widetilde{D}\widetilde{u}, (\widetilde{D}_X)^k \widetilde{u} \rangle
\]
on $\Omega$ whenever the latter is a contractible domain in $M$ and $(\widetilde{u}, \widetilde{D})$ is smooth and gauge equivalent to $(u, D)$ on $\Omega$.
\item[(b)] There exists a smooth function $\Phi_X$ on $M$ that restricts to
\[
|\widetilde{D}\widetilde{D}_X \widetilde{u} - \sqrt{-1}(\iota_X F)\widetilde{u}|^{2}
\]
on $\Omega$ whenever $\Omega$ and $(\widetilde{u}, \widetilde{D})$ are as in part (a).
\end{enumerate}
\end{lemm}
\begin{proof}
We only prove part (a) as (b) is similar. It suffices to show that if $\Omega$ is as in the statement then
\[
\alpha_{\Omega}:= \langle \widetilde{D}\widetilde{u}, (\widetilde{D}_X)^{k}\widetilde{u} \rangle
\]
is independent of the choice of $(\widetilde{u}, \widetilde{D})$. Indeed, if $(\widehat{u}, \widehat{D})$ is also smooth and gauge equivalent to $(u, D)$ on $\Omega$, then there exists $\theta \in W^{2, 2}(\Omega; \RR)$ such that 
\[
(\widehat{u}, \widehat{D}) = (e^{\sqrt{-1}\theta}\widetilde{u}, \widetilde{D} - \sqrt{-1}d\theta).
\]
Now recall by~\eqref{eq:YMH-gauge-rel} that 
\[
\widehat{D}_X \widehat{u} = e^{\sqrt{-1}\theta}\widetilde{D}_X\widetilde{u}.
\]
Applying $\widehat{D}_X$ to both sides, we get
\[
(\widehat{D}_{X})^2\widehat{u} = \widehat{D}_{X}\big(e^{\sqrt{-1}\theta}\widetilde{D}_X\widetilde{u} \big) = e^{\sqrt{-1}\theta}(\widetilde{D}_X)^2\widetilde{u}.
\]
Inductively, we obtain $(\widehat{D}_{X})^k\widehat{u} = e^{\sqrt{-1}\theta}(\widetilde{D}_X)^k\widetilde{u}$. Therefore, since $\langle \cdot, \cdot \rangle_{L}$ is Hermitian, we have
\[
\langle \widehat{D}\widehat{u}, (\widehat{D}_X)^{k}\widehat{u} \rangle = \big\langle e^{\sqrt{-1}\theta}\widetilde{D}\widetilde{u}, e^{\sqrt{-1}\theta}(\widetilde{D}_X)^{k}\widetilde{u} \big\rangle = \langle \widetilde{D}\widetilde{u}, (\widetilde{D}_X)^{k}\widetilde{u} \rangle,
\]
as asserted.
\end{proof}

\subsection{First and second inner variations of $F_{\ep}$}
Let $\Omega \subset M$ be a contractible domain, and let $\Omega_1$ be an open subset strictly contained in $\Omega$, with $\partial \Omega_1$ smooth. Suppose $(u, D = d - \sqrt{-1}A)$ is a smooth solution to~\eqref{eq:YMH-eq} on $\Omega$, and that $X$ is a vector field on $M$. Then the following integral is smooth in $t$:
\begin{align}\label{eq:YMH-inner-integral}
\nonumber F_{\ep}(\varphi_{t}^{\ast}u, d -\sqrt{-1}\varphi_{t}^{\ast}A;\Omega_1) &= \int_{\Omega_1}e_{\ep}(\varphi_{t}^{\ast}u, d -\sqrt{-1}\varphi_{t}^{\ast}A)d\mu_g,\\
&= \int_{\Omega_1} \ep^2|\varphi_{t}^{\ast}(dA)| + |\varphi_{t}^{\ast}(Du)|^2 + \frac{(1 - |\varphi_{t}^{\ast}u|^2)^2}{4\ep^2}d\mu_g. 
\end{align}
We define the first and second inner variations to be
\begin{equation}\label{eq:YMH-1st-inner}
\delta F_{\ep}(u, D; \Omega_1)(X) = \frac{d}{dt}F_{\ep}(\varphi_{t}^{\ast}u, d -\sqrt{-1}\varphi_{t}^{\ast}A; \Omega_1)\big\vert_{t =0}.
\end{equation}
\begin{equation}\label{eq:YMH-2nd-inner}
\delta^{2} F_{\ep}(u, D; \Omega_1)(X, X) = \frac{d^2}{dt^2}F_{\ep}(\varphi_{t}^{\ast}u, d -\sqrt{-1}\varphi_{t}^{\ast}A; \Omega_1)\big\vert_{t =0}.
\end{equation}
As in Section 2, we can relate the inner and outer variations as follows.
\begin{prop}\label{prop:YMH-inner-outer}
With $(u, D)$ and $\Omega_1$, $X$ as above, and writing $D = d - \sqrt{-1}A$, $F = dA$, we have 
\begin{enumerate}
\item[(a)]
\begin{equation}\label{eq:YMH-1st-inner-outer} 
\delta F_{\ep}(u, D)(X) = 2\int_{\partial\Omega_1} \langle \iota_{\nu}F_D, \iota_X F_D \rangle d\sigma_g + 2\re\int_{\partial\Omega_1}\langle D_{\nu}u, D_X u \rangle d\sigma_g,
\end{equation}
where $\nu$ denotes the outward unit normal to $\partial\Omega_1$.
\item[(b)]
\begin{align}
\nonumber& \delta^2 F_{\ep}(u, D)(X ,X)\\
\nonumber = &\int_{\Omega_1}2\ep^{2}|d(\iota_X F_D)|^{2} + 2|D D_X u - \sqrt{-1}(\iota_X F_D)u|^{2} - 4\re\langle \sqrt{-1}(\iota_X F_D)D_X u , Du \rangle\\
\nonumber& \frac{|u|^{2} - 1}{\ep^{2}}|D_X u|^{2} + \frac{2\big(\re\langle u, D_X u \rangle \big)^2}{\ep^2}d\mu_g\\
\label{eq:YMH-2nd-inner-outer} &+ 2\int_{\partial\Omega_1} \langle \iota_{\nu}F_D, \iota_X d(\iota_X F_D) \rangle d\sigma_g + 2\re\int_{\partial\Omega_1}\langle D_{\nu}u, D_X D_X u \rangle d\sigma_g.
\end{align}
\end{enumerate}
\end{prop}
\begin{rmk}\label{rmk:integrand-inv}
Notice that terms of the form $D D_X u$ and $D_X D_X u$ in (b) do not make sense unless $u$ has higher regularity than $W^{1, 2}$. This is why we work in domains over which the weak solution is gauge equivalent to a smooth one. In Section 6 we patch things together on $S^{n}$ with the help of Lemma~\ref{lemm:patching}.
\end{rmk}
Before giving the proof of Proposition~\ref{prop:YMH-inner-outer}, we single out some of the computations to be used for the proof of Proposition~\ref{prop:YMH-inner-outer} in the Lemma below.
\begin{lemm}\label{lemm:YMH-formulas}
In the notation of Proposition~\ref{prop:YMH-inner-outer}, we have:
\begin{enumerate}
\item[(a)] $\cL_X Du = D(D_X u) -\sqrt{-1}(\iota_X F)u + \sqrt{-1}(A_X) Du$.
\vskip 1mm
\item[(b)] 
\begin{align*}
\cL_X \cL_X Du =&\ D D_X D_X u - \sqrt{-1}(\iota_X d\iota_X F)u+ (\sqrt{-1}\nabla_X A_X - |A_X|^2)Du\\
& + 2A_X (\iota_X F)u -2\sqrt{-1}(\iota_X F)D_X u + 2\sqrt{-1}A_X D D_X u.
\end{align*}
\vskip 1mm
\item[(c)] Letting $\theta: \Omega \to \RR$ be a smooth function, then we have, pointwise,
\begin{equation}\label{eq:YMH-1st-gauge-var}
2\re\langle Du, \sqrt{-1}\theta Du \rangle + \frac{|u|^{2} - 1}{\ep^{2}}\re\big( u \overline{\sqrt{-1}\theta u} \big)= 0.
\end{equation}
\end{enumerate}
\end{lemm}
\begin{proof}
For (a) we recall that $d = D + \sqrt{-1}A$ and compute
\begin{align}
\nonumber \cL_{X}Du &= (d\iota_X + \iota_X d)Du\\
\label{eq:Lie-Du-middle} &= (D + \sqrt{-1}A)D_X u  + \iota_X (D + \sqrt{-1}A)Du.
\end{align}
In the second term, $D + \sqrt{-1}A$ acts like an exterior derivative, so
\[
(D + \sqrt{-1}A)Du = D^2 u + \sqrt{-1}A \wedge Du = -\sqrt{-1}Fu + \sqrt{-1}A \wedge Du.
\]
Thus, the last line in~\eqref{eq:Lie-Du-middle} becomes
\begin{align*}
&D(D_X u) + \sqrt{-1}A D_X u -\sqrt{-1}(\iota_{X}F )u + \sqrt{-1}\iota_X(A \wedge Du)\\
=&D(D_X u) -\sqrt{-1}(\iota_X F)u + \sqrt{-1}(A_X) Du.
\end{align*}

For part (b), we first use part (a) to get
\begin{equation}\label{eq:YMH-Lie-2}
\cL_{X}\cL_{X} Du = \cL_{X} D D_X u - \sqrt{-1}\cL_X \big( (\iota_X F)u \big) + \sqrt{-1}\cL_X \big( (A_X)Du \big).
\end{equation}
Applying part (a) with $D_X u $ in place of $u$ to the first term on the right-hand side yields
\begin{align*}
\cL_X DD_X u = D D_X D_X u - \sqrt{-1}(\iota_X F) D_X u + \sqrt{-1}A_X D D_X u.
\end{align*}
Substituting into~\eqref{eq:YMH-Lie-2} and performing some routine calculations, we get part (b).

Finally, part (c) holds because we are taking the real parts of purely imaginary numbers. 
\end{proof}
We are now ready to give the proof of Proposition~\ref{prop:YMH-inner-outer}.
\begin{proof}[Proof of Proposition~\ref{prop:YMH-inner-outer}]
We begin with (a). Differentiating~\eqref{eq:YMH-inner-integral} in $t$, we obtain
\begin{equation}\label{eq:YMH-1st-inner-diff}
\delta F_{\ep}(u, D; \Omega_1)(X) = \int_{\Omega_1} 2\ep^{2}\langle F, \cL_{X}F \rangle + 2\re\langle Du, \cL_{X}Du \rangle + \frac{|u|^{2} - 1}{\ep^{2}}\re\langle u, \cL_{X}u\rangle d\mu_g.
\end{equation}
Now recall the following facts:
\[
\cL_{X}F = (d\iota_X + \iota_X d)F = d(\iota_X F),
\]
\[
\cL_{X}u = \nabla_X u=  D_X u + \sqrt{-1}A_X u.
\]
Substituting these identities along with Lemma~\ref{lemm:YMH-formulas}(a) into~\eqref{eq:YMH-1st-inner-diff}, and using Lemma~\ref{lemm:YMH-formulas}(c) to eliminate some of the terms, we get
\begin{align}
\nonumber \delta F_{\ep}(u, D)(X) =& \int_{\Omega_1} 2\ep^{2}\langle F, d(\iota_X F) \rangle + 2\re\langle Du, D(D_X u) - \sqrt{-1}(\iota_X F)u \rangle\\
\label{eq:YMH-1st-inner-outer-1} & + \frac{|u|^{2} - 1}{\ep^{2}}\re\langle u, D_X u\rangle d\mu_g.
\end{align}
We now transform the above into a boundary integral. Testing~\eqref{eq:YMH-eq} against $(D_X u, \iota_X F)$ and integrating by parts over $\Omega_1$, we find that
\begin{align*}
0 =& \int_{\Omega_1} 2\ep^{2}\langle d^{\ast}F, \iota_X F \rangle - 2\langle u \times Du, \iota_X F \rangle + 2\re\langle D^{\ast}Du, D_{X}u \rangle + \frac{|u|^{2} - 1}{\ep^{2}}\re \langle u, D_X u\rangle d\mu_g\\
=& \int_{\Omega_1} 2\ep^{2}\langle F, d(\iota_X F) \rangle + 2\re\langle Du, D(D_X u) - \sqrt{-1}(\iota_X F)u \rangle+ \frac{|u|^{2} - 1}{\ep^{2}}\re \langle u, D_X u\rangle d\mu_g\\
&-2\int_{\partial\Omega_1} \langle \iota_{\nu}F, \iota_X F \rangle d\sigma_g - 2\re\int_{\partial\Omega_1}\langle D_{\nu}u, D_X u \rangle d\sigma_g.
\end{align*}
Combining this with~\eqref{eq:YMH-1st-inner-outer-1} gives~\eqref{eq:YMH-1st-inner-outer}.

To prove (b), we differentiate~\eqref{eq:YMH-inner-integral} once more to get
\begin{align}
\label{eq:YMH-2nd-inner-diff} \delta^{2}F_{\ep}(u, D; \Omega_1)(X, X)= & \int_{\Omega_1}2\ep^{2} |\cL_X F|^{2} + 2\ep^2\langle F, \cL_X \cL_X F \rangle \\
\nonumber&+ 2|\cL_X Du|^2 + 2\re\langle Du, \cL_X \cL_X Du \rangle\\
\nonumber& + \frac{|u|^{2} - 1}{\ep^2}\big( |\cL_X u|^{2} + \re\langle u, \cL_X \cL_X u \rangle \big) + \frac{2\re \langle u, \cL_X u \rangle^2}{\ep^2} d\mu_g.
\end{align}

Introducing the notation 
\[v = D_X u,\ a = \iota_X F,\ w = D_X D_X u,\ b = \iota_X d\iota_X F,
\]
we note that by Lemma~\ref{lemm:YMH-formulas}(b)(c), we have
\begin{align*}
&2|\cL_X Du|^{2} + 2\re\langle Du, \cL_X \cL_X Du \rangle\\
=&\ 2| Dv - \sqrt{-1}au+ \sqrt{-1}A_X Du |^{2} \\
&+ 2\re\langle Du, D w - \sqrt{-1}bu+ (\sqrt{-1}\nabla_X A_X - |A_X|^2)Du\rangle\\
&+  2\re\langle Du, 2(A_X) au -2\sqrt{-1}a v + 2\sqrt{-1}A_X Dv\rangle\\
=&\ 2|Dv - \sqrt{-1}au|^2 + 2|A_X|^2 |Du|^2 + 4\re\langle Dv - \sqrt{-1}au, \sqrt{-1}A_X Du\rangle\\
& + 2\re\langle Du, D w - \sqrt{-1}bu\rangle + 2\re\langle Du,(\sqrt{-1}\nabla_X A_X - |A_X|^2)Du \rangle\\
& -4\re\langle Du, \sqrt{-1}a v  \rangle + 4\re\langle Du, (A_X )au + \sqrt{-1}A_X Dv \rangle.
\end{align*}
Since $\sqrt{-1}A_X$ is purely imaginary, we have
\[
4\re\langle Dv - \sqrt{-1}au, \sqrt{-1}A_X Du\rangle + 4\re\langle Du, (A_X )au + \sqrt{-1}A_X Dv \rangle = 0.
\]
Moreover, by Lemma~\ref{lemm:YMH-formulas}(c) we have $\re\langle Du, \sqrt{-1}(\nabla_X A_X) Du \rangle = 0$. Consequently,
\begin{align*}
 2|A_X|^2 |Du|^2  + 2\re\langle Du,(\sqrt{-1}\nabla_X A_X - |A_X|^2)Du \rangle = 0.
\end{align*}
Thus we arrive at 
\begin{align*}
&\int_{\Omega_1}2|\cL_X Du|^{2} + 2\re\langle Du, \cL_X \cL_X Du \rangle d\mu_g\\
=&\ \int_{\Omega_1} 2|Dv - \sqrt{-1}au|^{2} -4\re\langle Du, \sqrt{-1}a v \rangle + 2\re\langle Du, Dw - \sqrt{-1}bu \rangle d\mu_g.
\end{align*}
A similar, but simpler, calculation shows that
\begin{align*}
&\int_{\Omega_1} \frac{|u|^{2} - 1}{\ep^2}\big( |\cL_X u|^{2} + \re\langle u, \cL_X \cL_X u \rangle \big) + \frac{2\re \langle u, \cL_X u \rangle^2}{\ep^2} d\mu_g\\
=&\ \int_{\Omega_1} \frac{|u|^2 - 1}{\ep^{2}}\big( |v|^{2} + \re\langle u, w \rangle \big) + \frac{2\re\langle u, v \rangle^2}{\ep^2}d\mu_g.
\end{align*}
Finally, it is straightforward to see that
\begin{align*}
&\int_{\Omega_1} 2\ep^{2} |\cL_X F|^{2} + 2\ep^2\langle F, \cL_X \cL_X F \rangle d\mu_g = \int_{\Omega_1} 2\ep^2 |da|^2 + 2\ep^{2}\langle F, db \rangle d\mu_g.
\end{align*}
Putting everything back into~\eqref{eq:YMH-2nd-inner-diff}, we get 
\begin{align*}
&\delta^2 F_{\ep}(u, D; \Omega_1)(X, X)\\
=&\ \int_{\Omega_1}2\ep^{2}|da|^{2} + 2|Dv - \sqrt{-1}au|^{2} -4\re\langle \sqrt{-1}av, Du \rangle + \frac{|u|^{2} - 1}{\ep^{2}}|v|^{2} + \frac{2\big(\re\langle u, v\rangle\big)^{2}}{\ep^{2}} d\mu_g\\
&+\int_{\Omega_1} 2\ep^2 \langle F, db \rangle + 2\re\langle Du, Dw - \sqrt{-1}bu \rangle + \frac{|u|^2 - 1}{\ep^2}\re\langle u, w \rangle d\mu_g.
\end{align*}
We are done upon repeating the end of the proof for part (a) to turn the second integral above into the following boundary integral:
\[
2\int_{\partial \Omega_1} \langle \iota_\nu F, b \rangle + \re\langle D_\nu u, w \rangle d\sigma_g.
\]
\end{proof}

\section{Stable critical points of $F_{\ep}$ on $S^{n}$}
Suppose $(u, D) \in \cC$ is a stable weak solution to~\eqref{eq:YMH-eq} on $S^{n}$ with $n \geq 2$. The following proposition relates stability as defined in~\eqref{eq:YMH-outer-stable} to stability with respect to inner variations with respect to the vector fields $X_\xi$. To emphasize the dependence of the YMH action density on the metric, we use the following notation:
\begin{align*}
e_{\ep}(u, D, g)d\mu_g &= \Big(\ep^2 |F_D|_{g}^{2} + |Du|_{g}^2 + \frac{(1 - |u|^2)^2}{4\ep^2}\Big)d\mu_g\\
&= \Big(\frac{1}{2}\ep^2 g^{ik}g^{jl}F_{ij}F_{kl} + g^{ij}\langle D_i u, D_j u \rangle + \frac{(1 - |u|^2)^2}{4\ep^2}\Big) \sqrt{\det(g)}dx.
\end{align*}

\begin{prop}\label{prop:YMH-inner-stable}
Let $\xi \in S^{n}$ and let $X = X_\xi$ be as in Proposition~\ref{prop:confKilling}, with $\varphi_{t}$ being the flow it generates. Let $g_t = \varphi_{-t}^{\ast}g$. If $(u, D)$ is a stable weak solution of~\eqref{eq:YMH-eq}, then 
\begin{enumerate}
\item[(a)]
\begin{equation}\label{eq:YMH-1st-inner-formula}
0= \int_{S^n} e_{\ep}(u, D)\Div(\nabla_{X} X) - 2\big( \ep^2 F_{e_i, e_k}F_{e_i, e_k} + \langle D_{e_i}u, D_{e_j}u \rangle \big)\langle \nabla_{i}(\nabla_X X), e_j \rangle d\mu_g.
\end{equation}
\item[(b)]
\begin{align}
\nonumber 0\leq& \int_{S^n} e_{\ep}(u, D)\big( (\Div X)^{2} - \Ric(X, X) - \langle \nabla_{e_{i}}X, e_{j} \rangle \langle \nabla_{e_{j}}X, e_{i} \rangle \big) d\mu\\ 
\nonumber & -4\int_{S^n} (\Div X) \langle \nabla_{i}X, e_j \rangle \big( \ep^2 F_{e_i, e_k}F_{e_j, e_k} + \langle D_{e_i}u, D_{e_j}u \rangle \big) d\mu_g\\
\nonumber & +2 \int_{S^n} \big(\langle R_{e_i, X}X, e_j\rangle -\langle \nabla_{i}X, \nabla_{j}X \rangle\big)\big( \ep^2 F_{e_i, e_k}F_{e_j, e_k} + \langle D_{e_i}u, D_{e_j}u \rangle \big)  d\mu_g\\
\nonumber & + \ep^{2}\int_{S^n}(\cL_X g)_{e_i, e_j} (\cL_X g)_{e_k, e_l} F_{e_i, e_k} F_{e_j, e_l} d\mu_g\\
\label{eq:YMH-inner-stable}  & + 2\int_{S^n} (\cL_X g)_{e_i, e_l} (\cL_X g)_{e_j, e_l}\big( \ep^2 F_{e_i, e_k}F_{e_j, e_k} + \langle D_{e_i}u, D_{e_j}u \rangle \big) d\mu_g.
\end{align}
\end{enumerate}
\end{prop}
\begin{proof}
Without loss of generality we assume $\xi = (1,0, \cdots, 0)$ and let $\Sigma = S^{n} \cap \{x^{n + 1} = 0\}$. Then both $X$ and $\nabla_{X} X\ ( = - f_\xi X)$ are tangent to $\Sigma$, and hence generate flows that preserve the upper and lower hemispheres, $S_{+}^n$ and $S_{-}^{n}$. Below we let $Y = \nabla_X X$ and let $\psi_{t}$ denote the flow of $Y$. The pullback $\psi_{-t}^{\ast}g$ will be denoted $h_t$. In addition, $p_{\pm} \in S_{\pm}^{n}$ will denote the north and south poles, respectively.

To prove part (a), we first fix some $\delta$ small and apply the considerations in Section 5.1 to $\Omega = S^{n}\setminus B_{\delta}(p^{-})$, $\Omega_1 = S^{n}_{+}$. Since $\Omega$ is contractible, there exists $\theta \in W^{2, 2}(\Omega; \RR)$ such that $(\widetilde{u},\widetilde{D}) = (e^{\sqrt{-1}\theta}u, D - \sqrt{-1}d\theta)$ is smooth in $\Omega$ and by Proposition~\ref{prop:YMH-inner-outer}(a) we have
\begin{equation}
\delta F_{\ep}(\widetilde{u}, \widetilde{D}; S^{n}_{+})(Y) = 2\int_{\Sigma} \langle \iota_{\nu}F_{\widetilde{D}}, \iota_Y F_{\widetilde{D}} \rangle d\sigma_g + 2\re\int_{\Sigma}\langle \widetilde{D}_{\nu}\widetilde{u}, \widetilde{D}_Y \widetilde{u} \rangle d\sigma_g.
\end{equation}
On the other hand, since $\psi_t$ preserves $S^{n}_{+}$, we find by the gauge invariance of $F_{\ep}$ that 
\begin{align}
\nonumber\int_{S^{n}_{+}} \ep^2|\psi_{t}^{\ast}F_{\widetilde{D}}|^2 + |\psi_{t}^\ast(\widetilde{D}\widetilde{u})|^2 + \frac{(1 - |\psi_{t}^{\ast}\widetilde{u}|^2)^2}{4\ep^2}d\mu_g&= \int_{S^{n}_{+}} e_{\ep}(\widetilde{u}, \widetilde{D}, h_t)d\mu_{h_t}\\
\label{eq:YMH-2nd-inner-pullback}&= \int_{S^{n}_{+}} e_{\ep}(u, D, h_t)d\mu_{h_t}.
\end{align}
Differentiating in $t$ and recalling the definition of $\delta F_{\ep}(\widetilde{u}, \widetilde{D}; S^{n}_{+})(Y)$, we get 
\begin{align*}
&\delta F_{\ep}(\widetilde{u}, \widetilde{D}; S^{n}_{+})(Y)\\
 = &-\int_{S^n} e_{\ep}(u, D)\Div Y - 2\big( \ep^2 F_{e_i, e_k}F_{e_j, e_k} + \langle D_{e_i}u, D_{e_j}u \rangle \big)\langle \nabla_{i}Y, e_j \rangle d\mu_g .
\end{align*}
To sum up, we have shown that 
\begin{align}\label{eq:YMH-1st-upper-boundary}
&\int_{S^n} e_{\ep}(u, D)\Div Y - 2\big( \ep^2 F_{e_i, e_k}F_{e_j, e_k} + \langle D_{e_i}u, D_{e_j}u \rangle \big)\langle \nabla_{i}Y, e_j \rangle d\mu_g\\
=& -2\int_{\Sigma} \langle \iota_{\nu}F_{\widetilde{D}}, \iota_Y F_{\widetilde{D}} \rangle + \re\langle \widetilde{D}_{\nu}\widetilde{u}, \widetilde{D}_Y \widetilde{u} \rangle d\sigma_g, \nonumber
\end{align}
where $\nu$ is the unit normal to $\Sigma$ that points downward. Similarly, on $S^{n}\setminus B_{\delta}(p_{+})$ we can find a real-valued $\phi$ in $W^{2, 2}$ such that, letting 
\[
(\widehat{u}, \widehat{D}) = (e^{\sqrt{-1}\phi}u, D - \sqrt{-1}d\phi),
\]
and with $\nu$ as in~\eqref{eq:YMH-1st-upper-boundary}, we have
\begin{align}\label{eq:YMH-1st-lower-boundary}
&\int_{S^n} e_{\ep}(u, D)\Div Y - 2\big( \ep^2 F_{e_i, e_k}F_{e_j, e_k} + \langle D_{e_i}u, D_{e_j}u \rangle \big)\langle \nabla_{i}Y, e_j \rangle d\mu_g\\
=& 2\int_{\Sigma} \langle \iota_{\nu}F_{\widehat{D}}, \iota_Y F_{\widehat{D}} \rangle + \re\langle \widehat{D}_{\nu}\widehat{u}, \widehat{D}_Y \widehat{u} \rangle d\sigma_g. \nonumber
\end{align}
In conclusion, the left-hand side of~\eqref{eq:YMH-1st-inner-formula} equals
\begin{equation}
2\int_{\Sigma} \langle \iota_{\nu}F_{\widehat{D}}, \iota_Y F_{\widehat{D}} \rangle + \re\langle \widehat{D}_{\nu}\widehat{u}, \widehat{D}_Y \widehat{u} \rangle d\sigma_g -2\int_{\Sigma} \langle \iota_{\nu}F_{\widetilde{D}}, \iota_Y F_{\widetilde{D}} \rangle + \re\langle \widetilde{D}_{\nu}\widetilde{u}, \widetilde{D}_Y \widetilde{u} \rangle d\sigma_g,
\end{equation}
which vanishes because $F_{\widetilde{D}} = F_{\widehat{D}}$ and because of Lemma~\ref{lemm:patching}(a).

To prove (b), differentiating twice the formula~\eqref{eq:YMH-2nd-inner-pullback} with $\varphi_t, g_t$ in place of $\psi_t, h_t$, respectively, we obtain
\begin{align}\label{eq:YMH-2nd-inner-pre}
\delta^2 F_{\ep}(\widetilde{u}, \widetilde{D}; S_{+}^n)(X, X) = \int_{S^{n}_{+}}\frac{d^2}{dt^2}e_{\ep}(u, D, g_t)\big|_{t=0}d\mu_{g_{t}}.
\end{align}
Using Proposition~\ref{prop:YMH-inner-outer}(b) to replace the left-hand side of~\eqref{eq:YMH-2nd-inner-pre}, we get that
\begin{align*}
&\int_{S^{n}_{+}}\frac{d^2}{dt^2}e_{\ep}(u, D, g_t)\big|_{t=0}d\mu_{g_{t}}\\
=&\int_{S_{+}^{n}}2\ep^{2}|d(\iota_X F_{\widetilde{D}})|^{2} + 2|\widetilde{D} \widetilde{D}_X \widetilde{u} - \sqrt{-1}(\iota_X F_{\widetilde{D}})\widetilde{u}|^{2} - 4\re\langle \sqrt{-1}(\iota_X F_{\widetilde{D}})\widetilde{D}_X \widetilde{u} , \widetilde{D}\widetilde{u} \rangle\\
&+ \frac{|\widetilde{u}|^{2} - 1}{\ep^{2}}|\widetilde{D}_X \widetilde{u}|^{2} + \frac{2\big(\re\langle \widetilde{u}, \widetilde{D}_X \widetilde{u} \rangle \big)^2}{\ep^2}d\mu_g\\
 &+ 2\int_{\Sigma} \langle \iota_{\nu}F_{\widetilde{D}}, \iota_X d(\iota_X F_{\widetilde{D}}) \rangle d\sigma_g + 2\re\langle \widetilde{D}_{\nu}\widetilde{u}, \widetilde{D}_X \widetilde{D}_X \widetilde{u} \rangle d\sigma_g\\
 =&\int_{S_{+}^{n}}2\ep^{2}|d(\iota_X F_{D})|^{2} + 2\Phi_X - 4\re\langle \sqrt{-1}(\iota_X F_{D})D_X u , Du \rangle + \frac{|u|^{2} - 1}{\ep^{2}}|D_X u|^{2} + \frac{2\big(\re\langle u, D_X u \rangle \big)^2}{\ep^2}d\mu_g\\
 &+ 2\int_{\Sigma} \langle \iota_{\nu}F_{\widetilde{D}}, \iota_X d(\iota_X F_{\widetilde{D}}) \rangle d\sigma_g + 2\re\langle \widetilde{D}_{\nu}\widetilde{u}, \widetilde{D}_X \widetilde{D}_X \widetilde{u} \rangle d\sigma_g,
\end{align*}
where $\Phi_X$ is as in Lemma~\ref{lemm:patching}(b). Repeating the argument on $S^{n}_{-}$, we get
\begin{align*}
&\int_{S^{n}_{-}}\frac{d^2}{dt^2}e_{\ep}(u, D, g_t)\big|_{t=0}  d\mu_{g_{t}}\\
 =&\int_{S_{+}^{n}}2\ep^{2}|d(\iota_X F_{D})|^{2} + 2\Phi_X - 4\re\langle \sqrt{-1}(\iota_X F_{D})D_X u , Du \rangle + \frac{|u|^{2} - 1}{\ep^{2}}|D_X u|^{2} + \frac{2\big(\re\langle u, D_X u \rangle \big)^2}{\ep^2}d\mu_g\\
 &- 2\int_{\Sigma} \langle \iota_{\nu}F_{\widehat{D}}, \iota_X d(\iota_X F_{\widehat{D}}) \rangle d\sigma_g + 2\re\langle \widehat{D}_{\nu}\widehat{u}, \widehat{D}_X \widehat{D}_X \widehat{u} \rangle d\sigma_g.
\end{align*}
Adding the computations on $S_{\pm}^{n}$ together, we arrive at 
\begin{align}
\nonumber&\int_{S^{n}}\frac{d^2}{dt^2}e_{\ep}(u, D, g_t)\big|_{t=0} d\mu_{g_{t}}\\
\nonumber =&\int_{S^{n}}2\ep^{2}|d(\iota_X F_{D})|^{2} + 2\Phi_X - 4\re\langle \sqrt{-1}(\iota_X F_{D})D_X u , Du \rangle + \frac{|u|^{2} - 1}{\ep^{2}}|D_X u|^{2} + \frac{2\big(\re\langle u, D_X u \rangle \big)^2}{\ep^2}d\mu_g\\
\nonumber &+ 2\int_{\Sigma} \langle \iota_{\nu}F_{\widetilde{D}}, \iota_X d(\iota_X F_{\widetilde{D}}) \rangle d\sigma_g + 2\re\langle \widetilde{D}_{\nu}\widetilde{u}, \widetilde{D}_X \widetilde{D}_X \widetilde{u} \rangle d\sigma_g\\
\label{eq:YMH-stable-computation} &- 2\int_{\Sigma} \langle \iota_{\nu}F_{\widehat{D}}, \iota_X d(\iota_X F_{\widehat{D}}) \rangle d\sigma_g + 2\re\langle \widehat{D}_{\nu}\widehat{u}, \widehat{D}_X \widehat{D}_X \widehat{u} \rangle d\sigma_g.
\end{align}
Again by Lemma~\ref{lemm:patching}, the last two lines cancel each other. On the other hand, by computations similar to those leading to~\eqref{eq:2nd-innervar} in Proposition~\ref{prop:2nd-var}, the first line equals the right-hand side of~\eqref{eq:YMH-inner-stable} plus the following term which vanishes by~\eqref{eq:YMH-1st-inner-formula},
\[
 \int_{S^n} e_{\ep}(u, D)\Div(\nabla_{X} X) - 2\big( \ep^2 F_{e_i, e_k}F_{e_j, e_k} + \langle D_{e_i}u, D_{e_j}u \rangle \big)\langle \nabla_{i}(\nabla_X X), e_j \rangle d\mu_g.
\]
Hence, the proof of inequality~\eqref{eq:YMH-inner-stable} is complete once we verify that the second line in~\eqref{eq:YMH-stable-computation} above is non-negative. This is the content of the Lemma below.
\end{proof}

\begin{lemm}\label{lemm:YMH-outer-global}
Suppose $(u, D) \in \cC$ is a stable weak solution to~\eqref{eq:YMH-eq} on $S^{n}$ with $|u| \leq 1$, and let $X$ be a vector field. Then
\begin{align}
\nonumber 0 \leq &\int_{S^{n}}2\ep^{2}|d(\iota_X F_{D})|^{2} + 2\Phi_X - 4\re\langle \sqrt{-1}(\iota_X F_{D})D_X u , Du \rangle\\
\label{eq:outer-global} &+ \frac{|u|^{2} - 1}{\ep^{2}}|D_X u|^{2} + \frac{2\big(\re\langle u, D_X u \rangle \big)^2}{\ep^2}d\mu_g.
\end{align}
\end{lemm}
\begin{proof}
Let $\Omega_{k}$ be an increasing sequence of (contractible) domains exhausting $S^{n}\setminus\{p^{-}\}$. Since $n \geq 2$ we may find cut-off functions $\zeta_k$ such that $\supp(\zeta_k) \subset \Omega_{k + 1}$, $\supp(1 - \zeta_k) \subset S^{n}\setminus \Omega_k$, and 
\begin{equation}\label{eq:capacity}
\lim_{k \to \infty}\int_{S^{n}} |d\zeta_k|^2d\mu_g = 0.
\end{equation}
For each $k$, choose $(\widetilde{u}, \widetilde{D})$ smooth and gauge equivalent to $(u, D)$, and let $X_k = \zeta_k X$. Then by stability and~\eqref{eq:YMH-2nd-outer},~\eqref{eq:YMH-outer-stable}, we have
\begin{align*}
0 \leq&\ \delta^{2}F_{\ep}(\widetilde{u}, \widetilde{D})(\widetilde{D}_{X_{k}}\widetilde{u}, \iota_{X_{k}}F_{\widetilde{D}})\\
=&\int_{S^{n}}2\ep^{2}|d(\iota_{X_k} F_{D})|^{2} + 2\Phi_{X_k} - 4\re\langle \sqrt{-1}(\iota_{X_k} F_{D})D_{X_k} u , Du \rangle\\
&+ \frac{|u|^{2} - 1}{\ep^{2}}|D_{X_k} u|^{2} + \frac{2\big(\re\langle u, D_{X_k} u \rangle \big)^2}{\ep^2}d\mu_g\\
=&\int_{S^{n}}2\ep^{2}\zeta_k^2|d(\iota_{X} F_{D})|^{2} + 2\zeta_k^2\Phi_{X} - 4\zeta_k^2\re\langle \sqrt{-1}(\iota_{X} F_{D})D_{X} u , Du \rangle\\
&+ \frac{|u|^{2} - 1}{\ep^{2}}\zeta_k^2|D_{X} u|^{2} + \frac{2\zeta_k^2\big(\re\langle u, D_{X} u \rangle \big)^2}{\ep^2}d\mu_g\\
&+\int_{S^n} 4\ep^{2}\langle  (\iota_{X}F_{D})d\zeta_k, \zeta_k d(\iota_{X}F_{D}) \rangle + 2\ep^2|\iota_X F_{D}|^2|d\zeta_k|^2d\mu_g\\
&+ \int_{S^n}4\Big\langle (\widetilde{D}_X \widetilde{u})d\zeta_k, \zeta_k (\widetilde{D}\widetilde{D}_{X}u - \sqrt{-1}(\iota_{X}F_{\widetilde{D}}) \widetilde{u}) \Big\rangle + 2 |D_{X}u|^2 |d\zeta_k|^2d\mu_g.
\end{align*}
To finish the proof it suffices to let $k \to \infty$ and make sure that the last two lines both tend to zero. Since $F_{D}$ is smooth on all of $S^{n}$, we have by~\eqref{eq:capacity} and H\"older's inequality that
\begin{align*}
\lim_{k \to \infty}\int_{S^n} 4\ep^{2}\langle  (\iota_{X}F_{D})d\zeta_k, \zeta_k d(\iota_{X}F_{D}) \rangle + 2\ep^2|\iota_X F_{D}|^2|d\zeta_k|^2d\mu_g = 0.
\end{align*}
On the other hand, note that 
\begin{align*}
 &\int_{S^n}4\Big\langle (\widetilde{D}_X \widetilde{u})d\zeta_k, \zeta_k (\widetilde{D}\widetilde{D}_{X}u - \sqrt{-1}(\iota_{X}F) \widetilde{u}) \Big\rangle + 2 |D_{X}u|^2 |d\zeta_k|^2d\mu_g\\
 =&\ \int_{S^{n}} 2\langle \zeta_k d\zeta_k, d(|D_{X}u|^2) - \sqrt{-1}(\iota_X F) \nabla_{X}|u|^2 \rangle + 2|D_{X}u|^2 |d\zeta_k|^2 d\mu_g.
\end{align*}
Since $|D_{X}u|^2$ and $|u|^2$ are smooth on all of $S^n$, we again see by~\eqref{eq:capacity} and H\"older's inequality that 
\[
\lim_{k\to \infty}\int_{S^n}4\Big\langle (\widetilde{D}_X \widetilde{u})d\zeta_k, \zeta_k (\widetilde{D}\widetilde{D}_{X}u - \sqrt{-1}\iota_{X}F \widetilde{u}) \Big\rangle + 2 |D_{X}u|^2 |d\zeta_k|^2d\mu_g = 0.
\]
The proof of Lemma~\ref{lemm:YMH-outer-global} is complete.
\end{proof}
It is now rather straightforward to prove Theorem~\ref{thm:YMH-no-stable}.
\begin{proof}[Proof of Theorem~\ref{thm:YMH-no-stable}]
We first show that $F_{D} = 0$, $Du = 0$ and $|u| \equiv 1$. As in the remarks after Proposition~\ref{prop:confKilling}, we let $\xi \in S^{n}$ and apply~\eqref{eq:YMH-inner-stable} with $X = X_{\xi}$. Then with the help of Proposition~\ref{prop:confKilling} we get
\begin{align*}
0 \leq  & \int_{S^n} e_{\ep}(u, D)\big( n^2 f_{\xi}^2 - (n-1)(1 - f_{\xi}^2) - n f_{\xi}^2 \big) d\mu_{g}\\
&-4\int_{S^n} n f_{\xi}^2 \delta_{ij} \big( \ep^{2}F_{e_i, e_k} F_{e_j, e_k} + \langle D_{e_i}u, D_{e_j} u \rangle \big) d\mu_{g}\\
& + 2\int_{S^n} \big( (1 - f_{\xi}^2)\delta_{ij} - \langle X_{\xi}, e_i \rangle\langle X_{\xi}, e_j \rangle - f^{2}_\xi \delta_{ij} \big)\big( \ep^2 F_{e_i, e_k} F_{e_j, e_k} + \langle  D_{e_i}u,  D_{e_j} u \rangle \big) d\mu_g\\
&+4\ep^2\int_{S^n} f_{\xi}^2 \delta_{ij}\delta_{kl}F_{e_i, e_k} F_{e_j, e_l} d\mu_{g}\\
& + 8\int_{S^n} f_{\xi}^2 \delta_{ij}\big( \ep^{2}F_{e_i, e_k} F_{e_j, e_k} + \langle D_{e_i}u, D_{e_j} u \rangle \big)d\mu_{g}.
\end{align*}
Next we substitute the standard basis vectors $\xi_{1}, \cdots, \xi_{n + 1}$ of $\RR^{n + 1}$ into the above inequality and add up the results. By a computation similar to that leading to~\eqref{eq:2nd-innvar-avg}, with additional help from the following identities
\begin{equation}\label{eq:curvature-identities}
\sum_{i, k = 1}^{n}F_{e_i, e_k}^{2} = 2|F|^2 \text{ and }\sum_{i = 1}^{n + 1}|\iota_{\xi_{i}}F|^{2} = 2|F|^2,
\end{equation}
we arrive at
\begin{equation}\label{eq:YMH-sphere-avg}
0 \leq 4(4 - n)\int_{S^{n}}\ep^2|F_D|^{2}d\mu_g + 2(2 - n)\int_{S^n}|Du|^2 d\mu_g.
\end{equation}
Note that the second identity in~\eqref{eq:curvature-identities} holds because the left-hand side is independent of the choice of orthonormal basis $\xi_1, \cdots, \xi_{n + 1}$ of $\RR^{n +1}$. Hence at each $x \in S^{n}$ we need only check the case where $\xi_1, \cdots, \xi_n$ is an orthonormal basis of $T_{x}S^n$ and $\xi_{n + 1} = x$, in which case the identity reduces to the first identity in~\eqref{eq:curvature-identities}.

When $n \geq 5$, inequality~\eqref{eq:YMH-sphere-avg} forces $F_{D}$ and $Du$ to both be identically zero. In other words, the connection $D$ is flat, and the section $u$ is covariantly constant. When $n = 4$,~\eqref{eq:YMH-sphere-avg} implies that $Du = 0$, but then from~\eqref{eq:YMH-eq} we have
\[
\ep^2 d^{\ast}F_{D} = \re\langle \sqrt{-1}u, Du \rangle = 0.
\]
Since $dF_{D} = 0$ as well, we conclude that $F_{D}$ is a harmonic $2$-form on $S^{4}$, and hence must vanish. Therefore we get that $D$ is flat and $u$ is covariantly constant when $n = 4$ also.

Since $S^{n}$ is simply-connected, the presence of a flat $U(1)$-connection on $L$ forces it to be isomorphic to $S^{n} \times \CC$ with the standard Hermitian metric. Note that we can still draw this conclusion even though $D$ may not be smooth globally. Indeed, since $D$ is of class $W^{1, 2}$, by definition there exists a real $1$-form $a$ in $W^{1, 2}$ and a smooth background connection $\widehat{D}$ on $L$ such that $D = \widehat{D} - \sqrt{-1}a$. Then $F_D = F_{\widehat{D}} + da$, which implies that $F_{\widehat{D}} = -da$, because $F_{D} = 0$. From this we claim that $F_{\widehat{D}} = d\alpha$ for a \textit{smooth} $1$-form $\alpha$. Indeed, consider the Hodge decomposition of the \textit{smooth} $2$-form $F_{\widehat{D}}$:
\[
F_{\widehat{D}} = d\alpha + d^{\ast}\beta,
\]
where $\alpha$ and $\beta$ are both smooth. (No harmonic part for $2$-forms on $S^n$ with $n \geq 3$.) Then
\[
\int_{S^n}|d^{\ast}\beta|^2 d\mu_g = \int_{S^n}\langle F_{\widehat{D}}, d^{\ast}\beta \rangle d\mu_g = -\int_{S^n}\langle da, d^{\ast}\beta \rangle d\mu_g = 0,
\]
where we integrated by parts to get the last equality, which is justified because $\beta$ is smooth. Hence $d^{\ast}\beta = 0$, and $F_{\widehat{D}} = d\alpha$ with $\alpha$ smooth, as claimed, but then $\widehat{D} + \sqrt{-1}\alpha$ defines a \textit{smooth} flat connection on $L$, and we can conclude that $L$ is trivial. 

We may then identify $u$ with a function $S^{n} \to\CC$, and write $D = d - \sqrt{-1}A$ for some real $1$-form $A$ on $S^n$. But since $dA = F_{D} = 0$, we must have 
\[A = -d\theta  \text{ for some } \theta: S^n \to \RR,
\]
in which case $Du = 0$ translates into $d(e^{\sqrt{-1}\theta}u) = 0$. It follows that $(u, D)$ is gauge equivalent to $(z_0, d)$ for some constant $z_{0} \in \CC$. 

The first equation in~\eqref{eq:YMH-eq} now implies that either $|z_0| = 1$ or $z_0 = 0$. To rule out the latter case we observe that for all smooth $v:S^{n} \to \CC$, 
\[
\delta^{2}F_{\ep}(0, d)((v, 0)) = \int_{S^{n}} 2|dv|^{2} - \frac{1}{\ep^{2}}|v|^{2} d\mu_{g}.
\]
In particular, taking $v = 1$ shows that $(0, d)$ is unstable. Thus we must have $|z_{0}| = 1$. From this we conclude that $(u, D)$ is in fact gauge equivalent to $(1, d)$.


\end{proof}
\bibliographystyle{amsalpha}
\bibliography{main_final}

\providecommand{\bysame}{\leavevmode\hbox to3em{\hrulefill}\thinspace}
\providecommand{\MR}{\relax\ifhmode\unskip\space\fi MR }
\providecommand{\MRhref}[2]{%
  \href{http://www.ams.org/mathscinet-getitem?mr=#1}{#2}
}
\providecommand{\href}[2]{#2}
\begin{thebibliography}{BGM71}

\bibitem[BBO01]{BBO}
F.~B\'ethuel, H.~Br\'ezis, and G.~Orlandi, \emph{Asymptotics for the
  {G}inzburg-{L}andau equation in arbitrary dimensions}, J. Funct. Anal.
  \textbf{186} (2001), no.~2, 432--520.

\bibitem[BGM71]{BGM}
Marcel Berger, Paul Gauduchon, and Edmond Mazet, \emph{Le spectre d'une
  vari\'{e}t\'{e} riemannienne}, Lecture Notes in Mathematics, Vol. 194,
  Springer-Verlag, Berlin-New York, 1971.

\bibitem[BL81]{BL}
Jean-Pierre Bourguignon and H.~Blaine Lawson, Jr., \emph{Stability and
  isolation phenomena for {Y}ang-{M}ills fields}, Comm. Math. Phys. \textbf{79}
  (1981), no.~2, 189--230.

\bibitem[Bra90]{Brad}
Steven~B. Bradlow, \emph{Vortices in holomorphic line bundles over closed
  {K}\"{a}hler manifolds}, Comm. Math. Phys. \textbf{135} (1990), no.~1, 1--17.

\bibitem[CH78]{CaHo}
Richard~G. Casten and Charles~J. Holland, \emph{Instability results for
  reaction diffusion equations with {N}eumann boundary conditions}, J.
  Differential Equations \textbf{27} (1978), no.~2, 266--273.

\bibitem[Che13]{ChK}
Ko-Shin Chen, \emph{Instability of {G}inzburg-{L}andau vortices on manifolds},
  Proc. Roy. Soc. Edinburgh Sect. A \textbf{143} (2013), no.~2, 337--350.

\bibitem[Che17]{thesis}
D.~R. Cheng, \emph{Geometric {V}ariational {P}roblems: {R}egular and {S}ingular
  {B}ehaviour}, Ph.D. thesis, Stanford {U}niversity, 2017.

\bibitem[Che20]{Che}
\bysame, \emph{Asymptotics for the {G}inzburg-{L}andau equation on manifolds
  with boundary under homogeneous {N}eumann condition}, J. Funct. Anal.
  \textbf{278} (2020), no.~4, 108364, 93pp.

\bibitem[GP94]{Gar}
Oscar Garc\'{\i}a-Prada, \emph{A direct existence proof for the vortex
  equations over a compact {R}iemann surface}, Bull. London Math. Soc.
  \textbf{26} (1994), no.~1, 88--96.

\bibitem[JM94]{JiMo}
Shuichi Jimbo and Yoshihisa Morita, \emph{Stability of nonconstant steady-state
  solutions to a {G}inzburg-{L}andau equation in higher space dimensions},
  Nonlinear Anal. \textbf{22} (1994), no.~6, 753--770.

\bibitem[JS02a]{JS}
Robert~L. Jerrard and Halil~Mete Soner, \emph{The {J}acobian and the
  {G}inzburg-{L}andau energy}, Calc. Var. Partial Differential Equations
  \textbf{14} (2002), no.~2, 151--191.

\bibitem[JS02b]{JiSt}
Shuichi Jimbo and Peter Sternberg, \emph{Nonexistence of permanent currents in
  convex planar samples}, SIAM J. Math. Anal. \textbf{33} (2002), no.~6,
  1379--1392.

\bibitem[KN96]{KN}
Shoshichi Kobayashi and Katsumi Nomizu, \emph{Foundations of differential
  geometry. {V}ol. {II}}, Wiley Classics Library, John Wiley \& Sons, Inc., New
  York, 1996, Reprint of the 1969 original, A Wiley-Interscience Publication.

\bibitem[Le15]{Le}
Nam~Q. Le, \emph{On the second inner variations of {A}llen-{C}ahn type energies
  and applications to local minimizers}, J. Math. Pures Appl. (9) \textbf{103}
  (2015), no.~6, 1317--1345.

\bibitem[Lee03]{Lee}
John~M. Lee, \emph{Introduction to smooth manifolds}, Graduate Texts in
  Mathematics, vol. 218, Springer-Verlag, New York, 2003.

\bibitem[LR99]{LR}
F.~H. Lin and T.~Rivi\`ere, \emph{Complex {G}inzburg-{L}andau equations in high
  dimensions and codimension-two area-minimizing currents}, J. Eur. Math. Soc.
  (JEMS) \textbf{1} (1999), no.~3, 237--311.

\bibitem[LS73]{LS}
H.~Blaine Lawson, Jr. and James Simons, \emph{On stable currents and their
  application to global problems in real and complex geometry}, Ann. of Math.
  (2) \textbf{98} (1973), 427--450.

\bibitem[Mat79]{Mat}
Hiroshi Matano, \emph{Asymptotic behavior and stability of solutions of
  semilinear diffusion equations}, Publ. Res. Inst. Math. Sci. \textbf{15}
  (1979), no.~2, 401--454.

\bibitem[Mor07]{Moro}
Andrei Moroianu, \emph{Lectures on {K}\"{a}hler geometry}, London Mathematical
  Society Student Texts, vol.~69, Cambridge University Press, Cambridge, 2007.

\bibitem[PS19]{PiSt}
A.~Pigati and D.~Stern, \emph{Minimal submanifolds from the abelian {H}iggs
  model}, arXiv:1905.13726 [math.DG] (2019).

\bibitem[RS75]{ReSi}
Michael Reed and Barry Simon, \emph{Methods of modern mathematical physics.
  {II}. {F}ourier analysis, self-adjointness}, Academic Press [Harcourt Brace
  Jovanovich, Publishers], New York-London, 1975.

\bibitem[Ser05]{Ser}
Sylvia Serfaty, \emph{Stability in 2{D} {G}inzburg-{L}andau passes to the
  limit}, Indiana Univ. Math. J. \textbf{54} (2005), no.~1, 199--221.

\bibitem[Sim68]{Sim}
James Simons, \emph{Minimal varieties in riemannian manifolds}, Ann. of Math.
  (2) \textbf{88} (1968), 62--105.

\bibitem[Ste]{Ste}
D.~Stern, \emph{Existence and limiting behavior of min-max solutions of the
  {G}inzburg-{L}andau equations on compact manifolds}, J. Differential Geom.,
  in press.

\bibitem[Xin80]{Xin}
Y.~L. Xin, \emph{Some results on stable harmonic maps}, Duke Math. J.
  \textbf{47} (1980), no.~3, 609--613.

\end{thebibliography}
\end{document}